\documentclass[11pt]{amsart}

\usepackage{enumerate, url, parskip, bm}
\usepackage{amssymb}
\usepackage{amsmath,amscd}
\usepackage{amsthm}
\usepackage{graphicx}
 \usepackage{verbatim}
 \usepackage{xcolor}
\usepackage{amsmath}
\usepackage[utf8]{inputenc}
\usepackage[T1]{fontenc}
\usepackage{hyperref}
\usepackage{diagrams}
\usepackage{cancel}
\usepackage{
  etex,  
  lmodern,
  calc,
  amssymb,
  amsthm,
  paralist,
  tabularx,
  multicol,
  graphicx,  
}

\usepackage{slashed} 
\allowdisplaybreaks
\usepackage{adjustbox}

\newcommand{\BigWedge}{\mathord{\adjustbox{valign=B,totalheight=.6\baselineskip}{$\bigwedge$}}}
            
\usepackage[all]{xy}

\usepackage{etoolbox}
\makeatletter
\patchcmd{\@maketitle}
  {\ifx\@empty\@dedicatory}
  {\ifx\@empty\@date \else {\vskip3ex \centering\footnotesize\@date\par\vskip1ex}\fi
   \ifx\@empty\@dedicatory}
  {}{}
\patchcmd{\@adminfootnotes}
  {\ifx\@empty\@date\else \@footnotetext{\@setdate}\fi}
  {}{}{}
\makeatother

\newcommand{\Bv}{\operatorname{BVec}} 

\newcommand{\RR}{\mathbb{R}}
\newcommand{\QQ}{\mathbb{Q}}
\newcommand{\CC}{\mathbb{C}}

\newcommand{\CP}{\mathbb{CP}}

\newcommand{\HH}{\mathbb{H}}
\newcommand{\ZZ}{\mathbb{Z}}

\newcommand{\rank}{\operatorname{rank}}

\newcommand{\bq}{/\!\!/}
\newcommand{\diag}{\operatorname{diag}}
\newcommand{\B}{(S^3)^2\bq T^2}

	\newcommand*{\defeq}{\mathrel{\vcenter{\baselineskip0.5ex \lineskiplimit0pt
				\hbox{\scriptsize.}\hbox{\scriptsize.}}}%
		=}

\newtheorem{theorem}{Theorem}[section]

\newtheorem{corollary}[theorem]{Corollary}
\newtheorem{lemma}[theorem]{Lemma}
\newtheorem{proposition}[theorem]{Proposition}

\newtheorem*{theorem*}{Theorem}

\theoremstyle{remark}
\newtheorem{remark}[theorem]{Remark}

\theoremstyle{definition}
\newtheorem{definition}[theorem]{Definition}

\newtheorem{example}[theorem]{Example}

\makeatletter
\def\equalsfill{$\m@th\mathord=\mkern-7mu
\cleaders\hbox{$\!\mathord=\!$}\hfill
\mkern-7mu\mathord=$}
\makeatother

\begin{document}

\abovedisplayskip=6pt plus3pt minus3pt
\belowdisplayskip=6pt plus3pt minus3pt

\title[The tangent bundle of biquotients]{Examples of biquotients whose tangent bundle is not a biquotient vector bundle}

\author[Michael Albanese]{Michael Albanese}
\address{University of Adelaide, South Australia, Australia}
\email{michael.albanese@adelaide.edu.au}

\author[Jason DeVito]{Jason DeVito}
\address{The University of Tennessee at Martin, Tennessee, USA}
\email{jdevito1@utm.edu}

\author[David Gonz\'alez-\'Alvaro]{David Gonz\'alez-\'Alvaro}
\address{Universidad Polit\'ecnica de Madrid, Spain}
\email{david.gonzalez.alvaro@upm.es}

\thanks{2010 \it  Mathematics Subject classification.\rm\ 
Primary 53C21; Secondary 57R22, 57T15.}

\thanks{The second author received support from NSF DMS-2105556. The third author received support from MINECO grant MTM2017-85934-C3-2-P (Spain).}

\begin{abstract}  A \textit{biquotient vector bundle} is any vector bundle over a biquotient $G\bq H$ of the form $G\times_{H} V$ for an $H$-representation $V$.  Over most biquotients, biquotient vector bundles are the only vector bundles known to admit metrics of non-negative sectional curvature, and hence they play a crucial role in the context of the converse to the Soul Theorem of Cheeger and Gromoll.  In this article, we study the question of when the tangent bundle of $G\bq H$ is a biquotient vector bundle. We find infinite families of examples of biquotients $M\cong G\bq H$ for which the tangent bundle is not a biquotient vector bundle for any presentation of $M$ as a biquotient.  In addition, we find infinite families of manifolds which arise as biquotients in two ways: one for which the tangent bundle is a biquotient bundle, and one for which it is not. Some of these results depend on an observation of Hirzebruch which relates the signature and Euler characteristic of an almost complex manifold. We include a proof of this fact as it seems to be missing from the literature.
\end{abstract} 

\maketitle

\thispagestyle{empty}

\section{Motivation and results}

In \cite{CG72}, Cheeger and Gromoll proved the Soul Theorem which asserts that a non-compact manifold $N$ equipped with a complete metric of non-negative sectional curvature must have a soul.  That is, $N$ must contain a totally convex closed submanifold $M$ for which $N$ has the structure of a vector bundle over $M$.  Being totally convex, the metric on $N$ restricts to a non-negatively curved metric on $M$.  The converse to the Soul Theorem asks which vector bundles over a closed non-negatively curved Riemannian manifold admit complete metrics of non-negative sectional curvature. The motivation for this article is to study the converse to the Soul Theorem for the tangent bundle $TM$ of a biquotient $M \cong G\bq H$ with $G$ compact.

Let us describe a certain class of vector bundles over a biquotient which will be considered in this article. Associated to any biquotient $G\bq H$ there is an $H$-principal bundle $H\rightarrow G\rightarrow G\bq H$. Each $H$-representation $V$ induces a vector bundle $G\times_H V:= (G\times V)/H$. We call a vector bundle of the form $G\times_H V$  a \textit{biquotient vector bundle} and denote the set of isomorphism classes of biquotient vector bundles as $\Bv(G\bq H)$.

The most relevant feature of biquotient vector bundles is that all of them carry metrics of non-negative sectional curvature. Indeed, as both $G$ and $V$ admit $H$-invariant non-negatively curved Riemannian metrics, the Gray-O'Neill formula \cite{Gr67,On1} implies that the submersion metric on $G\times_H V$ is non-negatively curved. Moreover, for most biquotients $G\bq H$, the only vector bundles known to admit non-negatively curved metrics are precisely those in $\Bv(G\bq H)$. Thus it is of crucial importance in the context of the converse to the Soul Theorem to determine which bundles over $G\bq H$ belong to $\Bv(G\bq H)$. The goal of this article is to investigate under what conditions the tangent bundle of a biquotient is isomorphic to a biquotient vector bundle.

Given a manifold $M$, we say that a biquotient $G\bq H$ is a \textit{presentation} for $M$ if there is a diffeomorphism $M\cong G\bq H$. Any biquotient admits infinitely many presentations $G_1\bq H_1\cong G_2\bq H_2\cong G_3\bq H_3\cong\dots$, so it is natural to ask how the sets $\Bv(G_i\bq H_i)$ are related to each other. It is well known that the tangent bundle of a homogeneous space $G/H$ is isomorphic to $G\times_H\mathfrak{p}$, where $\mathfrak{p}$ denotes the isotropy action of $H$. In the terminology of the present article, this means that $TM\in \Bv(G/H)$ for any homogeneous presentation $G/H$ of $M$. Our first result shows that this is no longer true for arbitrary biquotient presentations.

\begin{theorem}\label{thm:presentations_of_HP}  Suppose $M = \mathbb{C}P^k$ or $M = \mathbb{H}P^k$ and that $M\neq \mathbb{C}P^1$. Then $M$ admits a biquotient presentation $M \cong G\bq H$, with $G$ simply connected, for which $TM\notin \Bv(G\bq H)$.

\end{theorem}

As we show in Proposition \ref{prop:cover}, the set of isomorphism classes of  biquotient vector bundles can only be enlarged when passing to a cover of $G$.  Thus, the assumption that $G$ is simply connected in Theorem \ref{thm:presentations_of_HP} allows for the largest possible class of biquotient vector bundles.


Note that each of the manifolds $M$ in Theorem~\ref{thm:presentations_of_HP} admit presentations $G\bq H$ for which $TM\in\Bv(G/H)$: one can take any homogeneous presentation. This observation leads to two comments. First, Theorem~\ref{thm:presentations_of_HP} shows that whether or not $TM\in\Bv(G\bq H)$ can depend on the presentation $G\bq H$ of $M$ as a biquotient, and consequently the set $\Bv(G\bq H)$ depends on the presentation. Second, it might suggest that every biquotient has some presentation for which the tangent bundle is a biquotient vector bundle. It turns out that the latter does not hold for arbitrary biquotients, as our second result shows.

\begin{theorem}\label{thm:first_theorem} In each dimension $4,6$ and $16n+4\geq 20$, there is a simply connected closed biquotient $M$ for which $TM\notin \Bv(G\bq H)$ for any presentation $G\bq H$ of $M$ as a biquotient.  Moreover, in dimension 6, there are infinitely many homotopy types of such examples.

\end{theorem}

Given the very limited number of existing tools to endow non-compact manifolds with non-negatively curved metrics, Theorem~\ref{thm:first_theorem} indicates that new ideas are required to construct such a metric on the tangent bundle of an arbitrary biquotient.

Let us outline the proofs of the theorems above. The proof for the biquotients in Theorem~\ref{thm:presentations_of_HP} is very similar to that for the biquotients in dimensions $16n+4$ from Theorem~\ref{thm:first_theorem}. All these biquotients  have a presentation of the form $Spin(p)\bq (Spin(q)\times L)$, for certain $p>q$ and rank one Lie group $L$, see Section~\ref{S:examples_spin} for the concrete embeddings. These biquotients were discovered by Eschenburg \cite{Es84}, and Kapovitch and Ziller showed in \cite{KZ04} that those from Theorem~\ref{thm:first_theorem} have the same integral cohomology groups as $\HH P^{2k-1}$ but a different ring structure. The proof that the tangent bundle is not a biquotient bundle involves an analysis of the representations of the corresponding group $Spin(q)\times L$ together with the non-existence of almost complex structures on these spaces. The non-existence follows from a result of Hirzebruch, which we state as Theorem \ref{thm:Hirzebruch}. 

\begin{theorem}[Hirzebruch]\label{thm:Hirzebruch}
Suppose $M$ is a closed $4m$-dimensional manifold which admits an almost complex structure, then $\chi(M) \equiv (-1)^m\sigma(M) \bmod 4$.
\end{theorem}

The above statement first appeared in the commentary of Hirzebruch's collected works \cite[page 777]{HirCW}, where it was pointed out that the result follows from properties of the $\chi_y$ genus. We collect these properties and include a proof of Hirzebruch's theorem in Section~\ref{appendix}.


The 4 and 6-dimensional examples of Theorem~\ref{thm:first_theorem} are of the form $G\bq T$ where $G$ is simply connected and $T$ is a torus with $\rank T =\rank G$, and include the connected sum $\B\approx\CP^2\#\CP^2$ discovered by Totaro \cite {To02} and Eschenburg's positively curved inhomogeneous $6$-manifold $SU(3)\bq T^2$ \cite{Es84}. The proof uses the theory of characteristic classes. As 
Hepworth proved in his thesis \cite[Proposition~5.3.9]{He05}, for any biquotient of the form $G\bq T$ (with $T$ a torus of any rank), the tangent bundle is stably isomorphic to a biquotient vector bundle. We do not know whether Hepworth's result holds for arbitrary biquotients.  However we are able to give the following partial answer which only depends on the topology rather than the presentation.  

 
\begin{theorem}\label{thm:sufficient_conditions_stably_biquotient}
Let $M = G\bq H$ be a closed biquotient. If $\oplus_{i>0} H^{4i}(M,\QQ)=0$, then the tangent bundle $TM$ is stably isomorphic to a biquotient vector bundle $G\times_H V$ for some $H$-representation $V$. 
\end{theorem}

Theorem~\ref{thm:sufficient_conditions_stably_biquotient} follows as an easy combination of various results. We use work of Singhof \cite{Si93} together with our results from Section~\ref{sec:structure} to conclude that the tangent bundle is stably isomorphic to the Whitney sum of a biquotient bundle and the inverse of another biquotient bundle. The topological condition in Theorem~\ref{thm:sufficient_conditions_stably_biquotient} was found by the second two authors in the recent paper \cite{DG21} to ensure that every biquotient bundle has an inverse which is also a biquotient bundle.  Theorem~\ref{thm:sufficient_conditions_stably_biquotient} has the following implication for the converse to the Soul Theorem:

\begin{corollary}
For each closed biquotient $M = G\bq H$ with $\oplus_{i>0} H^{4i}(M,\QQ)=0$ there is some $k$ for which the product manifold $TM\times\RR^k$ admits a metric of non-negative sectional curvature.
\end{corollary}




Finally, we analyze the tangent bundle of simply connected closed biquotients $G\bq H$ of dimension at most $5$, which were classified by the second author and Pavlov \cite{De14, Pa04}. Recall that such a biquotient is diffeomorphic to one of the following spaces: $S^2$, $S^3$, $S^4$, $\CC P^2$, $S^2\times S^2$, $\mathbb{C}P^2\#\pm\mathbb{C}P^2$, $S^5$, $S^2\times S^3$, the Wu space $SU(3)/SO(3)$ or the non-trivial $S^3$-bundle over $S^2$, denoted  $S^3\,\widehat{\times}\, S^2$.


\begin{theorem}\label{thm:dims235}
Let $M$ be a simply connected closed biquotient of dimension at most $5$  other than $\mathbb{C}P^2\# \mathbb{C}P^2$, $S^4$ or $\mathbb{C}P^2$. Then $TM\in \Bv(G\bq H)$ for any reduced presentation of $M \cong G\bq H$ with $G$ simply connected.

On the other hand, if $M = S^4$ or $\mathbb{C}P^2$, then $TM$ is a biquotient bundle for some reduced presentations of $M \cong G\bq H$ with $G$ simply connected, and is not for others. If $M = \mathbb{C}P^2\#\mathbb{C}P^2$, then $TM\notin \Bv(G\bq H)$ for any presentation $M \cong G\bq H$.
\end{theorem}

A biquotient $G\bq H$ with $G = G_1\times G_2$ is called \textit{reduced} if $H$ does not act transitively on any simple factor of $G$.  Every biquotient has a reduced form. Among biquotient presentations, reduced presentations allow for the largest class of biquotient vector bundles, see Proposition \ref{prop:reduced}.



The positive results from Theorem~\ref{thm:dims235} imply that $TM$ admits a metric of non-negative sectional curvature for all simply connected closed biquotients $M$ of dimension at most $5$, except possibly when $M\cong\mathbb{C}P^2\#\mathbb{C}P^2$. This fact was already known: when $M$ is one of the two inhomogeneous examples $\mathbb{C}P^2\# - \mathbb{C}P^2$ and $S^3\, \widehat{\times}\, S^2$ one can use the existence of cohomogeneity one actions with codimension 2 singular orbits on $M$ in order to equip $TM$ with a non-negatively curved metric, see the articles by Grove and Ziller \cite{GroveZiller} and by Amann, Zibrowius and the third author \cite{AGZ} for details. In view of the current state of knowledge, $\mathbb{C}P^2\#\mathbb{C}P^2$ is the only known closed, simply connected, non-negatively curved manifold of dimension at most $5$ for which it is unknown whether its tangent bundle admits a non-negatively curved metric.

\textbf{Organization of the article. }In Section~\ref{sec:structure} we begin with some background information and then we prove some structure results regarding biquotient vector bundles. Specifically, we show that the largest class of biquotient vector bundles is obtained when $G$ is simply connected (Proposition~\ref{prop:cover}) and when $G$ is reduced (Proposition~\ref{prop:reduced}). Armed with these structure results, we prove Theorem~\ref{thm:sufficient_conditions_stably_biquotient} at the end of Section 2, while Section \ref{appendix} contains a proof of Theorem \ref{thm:Hirzebruch}. Section~\ref{S:examples_spin} is devoted to proving Theorem~\ref{thm:presentations_of_HP} and Theorem~\ref{thm:first_theorem} for dimensions of the form $16n+4$, with the remaining cases of Theorem~\ref{thm:first_theorem}, namely dimensions $4$ and $6$, appearing in Section~\ref{S:examples_torus}. Finally, in Section~\ref{S:low_dims} we prove Theorem~\ref{thm:dims235}.

\textbf{Acknowledgements. } We thank Richard Hepworth for helpful conversations, and Manuel Amann and Wolfgang Ziller for useful comments on an earlier draft of this article. We would also like to thank an anonymous referee for numerous helpful comments.


\section{Background and the set of biquotient vector bundles}\label{sec:structure}

\subsection{Definitions and notation}\label{SS:definitions}

Let us first recall the definition of a biquotient, following the approach by Totaro in \cite[Lemma~1.1~(3)]{To02}.

\begin{definition}\label{def:biq}  Let $G$ be a compact Lie group and $Z(G)$ its center, and let $Z\subseteq G\times G$ be the diagonal normal subgroup $Z\defeq \{(g,g) : g\in Z(G)\}$.  Any homomorphism $f\colon H\to (G\times G)/Z$ can be written in the form $f(h)=[f_1(h),f_2(h)]$ and determines a well-defined two-sided $H$-action $\star$ on $G$ by the rule $h\star g=f_1(h)gf_2(h)^{-1}$. When this action is free, the orbit space, denoted by $G\bq H$, inherits a manifold structure and is called a \emph{biquotient}.  

\end{definition}

When the homomorphism $f\colon H\to (G\times G)/Z$ is given by a subgroup inclusion $H\subseteq \{e\}\times G\subseteq G\times G$, the resulting space is homogeneous and denoted by $G/H$. 

More generally, when $H$ has the form $H=H_1\times H_2$ and the homomorphism $f\colon H\to (G\times G)/Z$ is given by subgroup inclusions $H_i\subseteq G$, we will often denote the resulting space by $H_1\backslash G/H_2$ instead of $G\bq (H_1\times H_2)$.

Suppose $G\bq H$ is a biquotient with $G = G_1\times G_2$ and that the projection of the $H$-action to $G_1$ is transitive.  Then, as shown by Totaro \cite[Lemma 3.3]{To02}, one can find another presentation for $G\bq H$ of the form $G_2\bq \hat{H}$ where $\hat{H}\subseteq H$.  This process can be repeated to obtain a presentation $\tilde{G}\bq \tilde{H}$ for which $\tilde{H}$ does not act transitively on any simple factor of $\tilde{G}$.

\begin{definition}  A biquotient $G\bq H$ is called \textit{reduced} if for any simple factor $G_i\subseteq G$, the projection of the $H$-action to $G_i$ is not transitive.
\end{definition}

By construction, associated to a biquotient $G\bq H$ there is an $H$-principal bundle $H\to G\to G\bq H$. Each representation $V$ of $H$ induces a vector bundle $G\times_H V$ over $G\bq H$, defined as the quotient of $G\times V$ via the diagonal action by $H$ consisting of the biquotient $H$-action on $G$ and the representation action on $V$; the projection map is given by $[g,v]\mapsto [g]$.

\begin{definition} A vector bundle over a biquotient $G\bq H$ which is isomorphic to one of the form $G\times_H V$ for an $H$-representation $V$ is called a \emph{biquotient vector bundle}, or just \emph{biquotient bundle} for simplicity. A biquotient vector bundle is called \textit{real} if $V$ is a real vector space and is called \textit{complex} if $V$ is a complex vector space.
\end{definition}

\begin{definition}  Given a biquotient $ G\bq H$, the set of isomorphism classes of biquotient vector bundles is denoted by $\Bv(G\bq H)$.
\end{definition}

As already mentioned in the Introduction, given a biquotient $G\bq H$, one can find other biquotients $G'\bq H'$ diffeomorphic to $G\bq H$. However, the notation $\Bv(G\bq H)$ always refers to bundles of the form $G \times_H V$.


\subsection{Structure results}
Our first structure result (Proposition~\ref{prop:cover} below) indicates that if we pull the biquotient structure $G\bq H$ back along a covering $\pi : G'\rightarrow G$, the set of biquotient vector bundles can only increase.

Suppose $Z' = \{(g,g)\in G'\times G': g\in Z(G')\}$ denotes the diagonal center of $G'\times G'$.  We note that the product map $\pi\times \pi:G'\times G'\rightarrow G\times G$  maps $Z'$ to $Z$ and hence descends to a covering $\psi:(G'\times G')/Z'\rightarrow (G\times G)/Z$. We include a proof for completeness.

\begin{proposition}  The map $\psi$ is a covering map.

\end{proposition}


\begin{proof} Let $\eta:G\times G\rightarrow (G\times G)/Z$ denote the natural projection and similarly define $\eta':G'\times G'\rightarrow (G'\times G')/Z'$.  We obviously have $\psi \circ \eta' = \eta \circ (\pi\times \pi)$.   It follows that $\eta'$ maps $\ker (\pi\times \pi)$ to $\ker \psi$.  Thus, we obtain the following commutative diagram:

\begin{diagram} \ker (\pi \times \pi) & \rTo& \ker \psi\\ \dTo & & \dTo\\ G'\times G' &\rTo^{\eta'} & (G'\times G')/Z'\\ \dTo_{\pi\times\pi} & & \dTo_{\psi} \\
G\times G&\rTo^{\eta} & (G\times G)/Z.
\end{diagram}

We claim that $\eta'|_{\ker (\pi\times \pi)}:\ker (\pi\times \pi)\rightarrow \ker \psi$ is a surjective map of Lie groups.  Believing this claim, since $\ker (\pi\times \pi)$ is discrete, it will follow that $\ker \psi$ is discrete, so $\psi$ is a covering.

To see that $\eta'|_{\ker (\pi \times \pi)}:\ker (\pi\times \pi)\rightarrow \ker \psi$ is surjective, let $(a,b)Z'\in \ker \psi$.  Thus, $(\pi\times \pi)(a,b)\in Z$.  Since $\pi \times \pi$ is surjective, it maps $Z'$ surjectively onto $Z$, so there is an element $(z',z')\in Z'$ with $(\pi\times \pi)(a,b) = (\pi\times \pi)(z',z')$.  Then $(a,b)(z',z')^{-1}\in \ker (\pi\times \pi)$ and $\eta'((a,b)(z',z')^{-1}) = (a,b)Z'$.
\end{proof}

Next set $H' = \psi^{-1}(f(H))$, where $f:H\to (G\times G)/Z$ denotes the homomorphism defining the biquotient $G\bq H$. Since both $f$ and $\psi$ are group homomorphisms, it follows that $H'$ is a subgroup of $(G'\times G')/Z'$, and the corresponding inclusion map induces a free biquotient action of $H'$ on $G'$, denoted by $\star'$. Moreover, $\pi$ induces a diffeomorphism $G'\bq H'\rightarrow G\bq H$ (as shown by Totaro \cite[Lemma~3.1]{To02}). With this in mind, we have the following proposition.



\begin{proposition}\label{prop:cover} Let $G\bq H$ be a biquotient, $\pi:G'\rightarrow G$ a covering map, and $H'$ as defined above so that $G\bq H$ is diffeomorphic to $G'\bq H'$. Then  $\Bv(G\bq H)\subseteq \Bv(G'\bq H')$.  That is, every vector bundle of the form $G\times_H V$ is isomorphic to a vector bundle of the form $G'\times_{H'} V'$.

\end{proposition}

\begin{proof}
Let $\rho:= \psi|_{H'}:H'\rightarrow H$ denote the projection.  Given a representation $V$ of $H$, we let $V'$ be the same vector space endowed with the $H'$-action given by precomposing the $H$-action with $\rho$.  That is, $h'\in H'$ acts via the action $\bullet'$ defined by the rule $h'\bullet' v = \rho(h')\bullet v$, where $\bullet$ denotes the $H$-action on $V$.


Then we claim that the map $\pi\times Id_V:G'\times V'\rightarrow G\times V$ descends to a bundle isomorphism $\phi:G'\times_{H'} V'\rightarrow G\times_H V$.  Note first that $\phi$ covers the diffeomorphism $ G'\bq H'\rightarrow G\bq H$ induced by $\pi:G'\rightarrow G$, so $\phi$ is a bundle map, if it is well defined.

We observe that $\phi$ is well defined since:
\begin{align*}\phi(h'(g',v)H')&= \phi((h'\star' g', h'\bullet' v)H')\\ &= ((\pi\times Id_V)(h'\star'g', h'\bullet' v))H \\\ &= (\rho(h')\star\pi(g'), \rho(h')\bullet v)H\\ &= \rho(h')(\pi(g'),v)H\\ &= (\pi(g'),v)H \\ &= \phi((g',v)H').\end{align*}
Because the composition $G'\times V'\rightarrow G\times V\rightarrow G\times_H V$ is a surjective submersion, the same is true of $\phi$.  Thus, we only need to show it is injective.

To that end, first note that if $\phi( (g_1^\prime, v_1) H') = \phi((g_2^\prime,v_2)H')$, then $(g_1^\prime,v_1)H'$ and $(g_2^\prime,v_2)H'$ must be in the same fiber above $G'\bq H'$.  In particular, there is $h'\in H'$ with $h'\star g_1^\prime = g_2^\prime$.  Applying $h'$ to $(g_1^\prime,v_1)$, we obtain a new representative $(g_2^\prime, w)H'$ of the orbit through $(g_1^\prime,v_1)$.

Now, $\phi((g_2^\prime, w)H') = \phi((g_2^\prime,v_2)H')$, so $(\pi(g_2^\prime), w)H = (\pi(g_2^\prime), v_2)H$.  Thus, there is $h\in H$ with $h(\pi(g_2^\prime), w) = (\pi(g_2^\prime), v_2)$.  In particular, $h\star \pi(g_2^\prime) = \pi(g_2^\prime)$.  Since the $H$-action on $G$ is free, $h = e$.  In particular, $w = v_2$, so $\phi$ is injective.
\end{proof}


It is natural to wonder about the converse operation - passing from biquotient vector bundles on $G'\bq H'$ to biquotient vector bundles on $G \bq H$.  The following proposition shows that this is not always possible.

\begin{proposition} For the two presentations $$S^2 \cong SO(3)/SO(2) \cong Spin(3)/Spin(2) = SU(2)/U(1),$$ we have a strict inclusion $\Bv(SO(3)/SO(2))\subsetneq \Bv(SU(2)/U(1))$.
\end{proposition}

\begin{proof}  Suppose $E\rightarrow S^2$ is the rank $2$ real vector bundle over $S^2$ with Euler class $\pm 1$.  We will show that $E\in \Bv(SU(2)/U(1))$ but $E\notin \Bv(SO(3)/SO(2))$. 

We first show that $E\in \Bv(SU(2)/U(1))$.  Let $V= \mathbb{R}^2$ with $U(1)$ acting as a unit speed rotation.  We claim that $SU(2)\times_{U(1)} V$ has Euler class $\pm 1$.  To see this, restrict $V$ to the unit sphere $S^1\subseteq V$.  The action of $U(1)$ on $S^1$ is simply transitive, so we find that $SU(2)\times_{U(1)} S^1\cong SU(2)$.  From the Gysin sequence applied to the bundle $S^1\rightarrow SU(2)\rightarrow S^2$, we see $e = \pm 1$.

Next, we show that $E\notin \Bv(SO(3)/SO(2))$.  The representations of $SO(2)$ on $V\cong \mathbb{R}^2$ are all given by rotation by some speed $k$.  If $k = 0$, we obviously get the trivial bundle, which has Euler class $0$.  Thus, we may assume $k\neq 0$.  Now, consider the unit sphere bundle $SO(3)\times_{SO(2)} S^1$.  Because $k\neq 0$, the $SO(2)$-action on $S^1$ is transitively with stabilizer given by the $k$-th roots of $1$.  But then, as shown by Kapovitch and Ziller \cite[Lemma 1.3]{KZ04}, we find that $SO(3)\times_{SO(2)}S^1\cong SO(3)/\mathbb{Z}_k$.  In particular, the unit sphere bundle cannot be simply connected.
\end{proof}

\begin{remark}  In fact, it is easy to see that every vector bundle over $S^2 = SU(2)/U(1)$ is a biquotient (homogeneous) vector bundle. In contrast, using the presentation $S^2 = SO(3)/SO(2)$, the corresponding biquotient bundles are precisely the vector bundles with vanishing second Stiefel-Whitney class (including $TS^2$).
\end{remark}

Recall that every connected compact Lie group $G$ has a cover of the form $G'\times T$ where $G'$ is simply connected and $T$ is a torus.  Moreover, a simply connected Lie group is isomorphic to a product of simply connected simple Lie groups.  As Proposition \ref{prop:cover} indicates, passing to a cover can only enlarge the class of biquotient vector bundles.  As such, we will always assume that $G$ has the form $G'\times T$.

Having determined how the set $\Bv(G\bq H)$ changes when passing to covers of $G$, we now investigate how the set $\Bv(G\bq H)$ changes when passing to a reduced presentation.  We recall that Totaro \cite[Lemma~3.3]{To02} has shown that every biquotient $M=G\bq H$ has a reduced presentation. We outline his argument because we will need some relevant notation below.  

Suppose $G = G_1\times G_2$.  Then for each $i\in \{1,2\}$ we have a projection map $p_i:(G\times G)/Z\rightarrow (G_i\times G_i)/Z_i$, where $Z_i = \{(g,g)\in G_i\times G_i: g\in Z(G_i)\}$. The composition $p_1\circ f:H\rightarrow (G_1\times G_1)/Z_1$ defines a biquotient action of $H$ on $G_1$.  Suppose this action of $H$ on $G_1$ is transitive and let $\hat{H}$ denote the isotropy subgroup at the identity $e\in G_1$.  Then using the other projection $p_2:(G\times G)/Z\rightarrow (G_2\times G_2)/ Z_2$, one obtains a biquotient action of $\hat{H}$ on $G_2$.  Totaro \cite[Lemma~3.3]{To02} shows that the $\hat{H}$-action on $G_2$ is free if the $H$-action on $G_1\times G_2 $ is free and that the natural inclusion $G_2\rightarrow \{e\}\times G_2\subseteq G$ induces a diffeomorphism from $G\bq H$ to $G_2\bq \hat{H}$.

The next proposition indicates that the reduced form of a biquotient gives the largest possible class of biquotient vector bundles.


\begin{proposition}\label{prop:reduced}  In the notation of the above discussion, suppose a biquotient $(G_1\times G_2)\bq H$ is not reduced, with $(G_1\times G_2)\bq H \cong G_2\bq \hat{H}$.  Then $$\Bv( (G_1\times G_2)\bq H)\subseteq \Bv( G_2\bq \hat{H}).$$

\end{proposition}

\begin{proof}Suppose $E\in \Bv((G_1\times G_2)\bq H)$ and write $E\cong (G_1\times G_2)\times_H V$ for an $H$-representation $V$.  Let $W$ denote $V$ with the action restricted to $\hat H$.   Then, as shown by Totaro \cite[Lemma~3.3]{To02}, the inclusion map $G_2\times W\rightarrow (G_1\times G_2)\times V$ induces a diffeomorphism $G_2\times_{\hat{H}} W \cong (G_1\times G_2)\times_{H} V.$  It remains to see that it respects the bundle projections.

{\color{black} But this is clear because the bundle projections are induced from the canonical projections $G_2\times W\rightarrow G_2$ and $(G_1\times G_2)\times V\rightarrow (G_1\times G_2)$, and the inclusion $G_2\times W\rightarrow (G_1\times G_2)\times V$ obviously respects these projection maps.} 
\end{proof}

It is natural to wonder about the converse of Proposition \ref{prop:reduced}. In Proposition~\ref{prop:strict_inclusion} below, we provide an example which illustrates that the inclusion $\Bv((G_1\times G_2)\bq H) \subseteq \Bv(G_2\bq \hat{H})$ can be strict.  To describe the example, we let $\iota:Sp(2)\rightarrow SU(4)$ with $\iota(A)=\iota(A_1+A_2j) = \begin{bmatrix} A_1 & A_2\\ -\overline{A}_2 & \overline{A}_1\end{bmatrix}$ be the inclusion obtained from the identification of $\mathbb{H}^2$ with $\mathbb{C}^4$.

We let $H:=Sp(2)\times SU(3)$ and $G = G_1\times G_2$ where $G_1 = SU(4)$ and $G_2 = Sp(3)\times SU(4)$.  Then $H$ acts on $G$ by the formula $$(A,B)\star (C,(D,E)) = (\iota(A)C \diag(B,1)^{-1}, (\diag(A,1) D, \diag(B,1) E)).$$

The projection of the $H$-action to $G_2$ is free, hence the $H$-action on $G$ is free.  Let $$\hat{H} = \left\{\left(\begin{bmatrix} a+bj & 0 \\ 0 &1\end{bmatrix}, \begin{bmatrix} a& b & 0 \\ -\overline{b} & \overline{a} & 0\\ 0 &0 & 1\end{bmatrix}\right) \in H: a,b\in \mathbb{C} \text{ and }|a|^2 + |b|^2 = 1\right\}$$ and observe that $\hat{H}\cong SU(2)$.

\begin{proposition}\label{prop:strict_inclusion}
Suppose $G$, $H$ and $\hat{H}$ are as above.  The biquotient $G\bq H$ admits a reduced presentation $G_2\bq \hat{H}$.   In addition, if $V$ is the standard representation of $SU(2)$, then the biquotient vector bundle $G_2\times_{\hat{H}} V\in \Bv(G_2\bq \hat{H})$ is not isomorphic to any biquotient vector bundle in $\Bv(G\bq H)$.
\end{proposition}

\begin{proof}  Consider the projection of the $H$-action to $G_1$.   Note that $\iota(Sp(2))\backslash SU(4) = Spin(5)\backslash Spin(6) \cong S^5$.  Under this diffeomorphism, the action of the other factor $SU(3)$ of $H$ on $S^5$ is the standard action, so it is transitive.  This implies that projection of the $H$-action to $G_1$ is transitive.  An easy computation reveals that $\hat{H}$ is the isotropy group at the identity in $G_1$, so we have a diffeomorphism $G\bq H\cong G_2\bq \hat{H}$.  Since this $\hat{H}$-action on $G_2$ is obviously not transitive on either factor of $G_2$, this is a reduced form.

\bigskip

We next claim that $E:=G_2\times_{\hat{H}} V$ is a non-trivial vector bundle.  To see this, note first that if it were trivial, then the unit sphere bundle $E^1$ would be diffeomorphic to $(G_2\bq \hat{H}) \times S^3$.  In particular, the homotopy group $\pi_6(E^1)$ would be non-trivial owing to the fact that  $\pi_6(S^3) \cong\mathbb{Z}/12\mathbb{Z}$ is non-trivial.  However, we will show that $\pi_6(E^1) = 0$.  

To see this, simply note that $E^1 \cong G_2\times_{\hat{H}} S^3$ where $\hat{H}\cong SU(2)$ acts on $S^3$ via the standard transitive action.  In particular, $G_2\times_{\hat{H}} S^3\cong G_2 = Sp(3)\times SU(4)$.  For both $Sp(3)$ and $SU(4)$, $\pi_6$ is in the stable range, so $\pi_6(Sp(3)) = \pi_6(SU(4)) = 0$ by Bott periodicity.  This completes the proof that $E$ is non-trivial.

\bigskip

Finally, we prove that $E$ is not isomorphic to a vector bundle in $\Bv(G\bq H)$ by showing that all rank $4$ vector bundles in $\Bv(G\bq H)$ are trivial.  Indeed, suppose $W$ is a $4$-dimensional real representation of $H$.  Then this is defined by a map $H\rightarrow O(4)$.  Both simple factors of $H$ have dimension larger than $6 = \dim O(4)$, so the map must be trivial.  This implies that $G\times_{H} W$ is trivial.
\end{proof}

\begin{remark}There are infinitely many other possibilities for $G_2 = Sp(3)\times SU(4)$  in the proof above.  All that is needed is that $H$ acts freely on $G_2$, and $\pi_6(G_2)$ does not contain an isomorphic copy of $\pi_6(S^3)\cong \mathbb{Z}/12\mathbb{Z}$.

\end{remark}


\subsection{Alternative biquotients related to $G\bq H$} 

In the literature, one can find a slightly different definition of biquotient where a homomorphism $H'\rightarrow G\times G$ is used instead of a homomorphism to $(G\times G)/Z$. In this other setting, one is allowed more generally to consider effectively free actions, say with ineffective kernel $K \subseteq H'$.  

It is easy to translate from one definition to the other, at least when $Z$ is finite.  Specifically, if \textcolor{black}{$f: H\rightarrow  (G\times G)/Z$} defines a free biquotient action, then we may lift $f$ to a map $f':H'\rightarrow G\times G$ for some cover $H'$ of $H$.  It is easy to see that $f'$ defines an effectively free action of $H'$ on $G$ and that the orbit spaces $G\bq H$ and $G\bq H'$ are diffeomorphic.  Moreover, if $K\subseteq H'$ is the kernel of the $H'$-action on $G$, then $H'/K\cong H$.
  
Conversely, given $f':H'\rightarrow G\times G$ defining an effectively free action of $H'$ on $G$, let $K$ denote the kernel of the composition $H'\xrightarrow{f'} G\times G\rightarrow (G\times G)/Z$.  Then the composition defines a free biquotient action of $H:=H'/K$ on $G$ and the two orbit spaces $G\bq H'$ and $G\bq H$ are diffeomorphic.

Using the alternative definition of biquotient, one must exercise caution when forming vector bundles in the form $G\times_{H'} V$. Specifically, given a representation $V$ of $H'$, the space $G\times_{H'} V$ need not be a smooth manifold owing to the fact that the isotropy group at a point in the zero-section contains $K$, while the isotropy group at a non-zero vector is in general smaller.  Because of this, when using the alternative definition, one must restrict to representations of $H'$ for which $K$ acts trivially. 

In other words, in this alternative definition we do not, in general, get an $H'$-principal bundle $H'\to G\to G\bq H'$.  Rather, we obtain an $H'/K$-principal bundle $H'/K\to G\to G\bq H'$; thus the construction $G\times_{H'} V$ is only guaranteed to create a vector bundle over $G\bq H'$ if $V$ is a representation of $H'/K$.

{\color{black}  In the case of non-effective $H'$-actions on $G$ with kernel $K$, it is natural to wonder how the set of isomorphism classes of bundles of the form $G\times_{H'} V$ (with $V$ a representation of $H'/K$) compares with $\Bv(G\bq H)$.  The next proposition shows that these two sets coincide when $Z(G)$ is finite.}

\begin{proposition}  Suppose $Z=\Delta Z(G)\subseteq G\times G$ is finite.  Suppose $G\bq H$ is a biquotient and that $G\bq H'$ is an alternative biquotient related to $G\bq H$ as above.  Then the set $\Bv(G\bq H)$ coincides with the set of isomorphism classes of vector bundles of the form $G\times_{H'} V'$, where $V'$ is an $H'$-representation for which $K$ acts trivially.

\end{proposition}

\begin{proof} 

Suppose $V'$ is an $H'$-representation and that $K$ acts trivially. Let $V$ denote $V'$ with action by $H = H'/K$.  Then $K$ acts trivially on $G\times V'$, so $G\times_{H'} V' = G\times_{H'/K} V'\cong G\times_H V$.

Conversely, given an $H$-representation $V$, let $V'$ denote $V$ with $H'$ acting via the projection from $H'$ to $H$.
\end{proof}

{\color{black}  For the duration of the paper, the word ``biquotient'' will only refer to the definition given in Definition \ref{def:biq}.}




\subsection{The stable class of the tangent bundle}

Let $G\bq H$ be a closed biquotient defined by a homomorphism $f$. As shown by Eschenburg \cite[Section 36]{Es84}, $G\bq H$ can be given the presentation $G\backslash(G\times G)/H$, where $G$ (resp. $H$) is embedded in $G\times G$ diagonally (resp. via the homomorphism $f$). 

Let $\alpha_H$ denote the biquotient vector bundle (with respect to the presentation $G\bq H$) $G\times_H \mathfrak{h}$ with $\mathfrak{h}$ denoting the Lie algebra of $H$, with $H$ acting via the adjoint action.  Let $\alpha_G$ denote the vector bundle $((G\times G)/H) \times_G \mathfrak{g}$, defined similarly. Building upon work of Singhof \cite{Si93}, Kerin shows in \cite[Lemma~6.2]{Ke11} that the tangent bundle $T(G\bq H)$ of a biquotient has the property that 
$$T(G\bq H)\oplus \alpha_H\cong \alpha_G.$$
Rewriting $\alpha_G$ as $(G\times G)\times_{G\times H} \mathfrak{g}$, where the $H$ factor of $G\times H$ acts trivially on $\mathfrak{g}$, we see that it can be regarded as a biquotient vector bundle with respect to the presentation $G\backslash(G\times G)/H$. By Proposition~\ref{prop:reduced}, the bundle $\alpha_G$ is isomorphic to a biquotient vector bundle of the form $G\times_H \mathfrak{g}$ for some action of $H$ on $\mathfrak{g}$. 

{\color{black} With this background, we are ready to prove Theorem \ref{thm:sufficient_conditions_stably_biquotient}, which we restate for the readers' convenience.

\begin{theorem}
Let $M = G\bq H$ be a closed biquotient. If $\oplus_{i>0} H^{4i}(M,\QQ)=0$, then the tangent bundle $TM$ is stably isomorphic to a biquotient vector bundle $G\times_H V$ for some $H$-representation $V$. 
\end{theorem}}

\begin{proof}
It follows from the discussion above that $T(G\bq H)$ is stably isomorphic to $\alpha_G \oplus E$, where $E$ denotes any inverse of $\alpha_H$ (i.e.~any vector bundle such that the sum $E\oplus \alpha_H$ is isomorphic to a trivial vector bundle). From \cite[Theorem~1.3]{DG21} we know that the topological assumption implies that there an inverse $E_0$ which is isomorphic to a biquotient vector bundle $G\times_H V$. Since the map taking $H$-representations to biquotient vector bundles over $G\bq H$ is closed under taking Whitney sums, it follows that $\alpha_G \oplus E_0$ is a biquotient vector bundle.
\end{proof}

\begin{remark}
Suppose $G\bq H = G/H$ is homogeneous. When using Proposition~\ref{prop:reduced} to rewrite $\alpha_G=((G\times G)/H) \times_G \mathfrak{g}=(G\times (G/H))\times_G \mathfrak{g}$ as a bundle of the form $G\times_H \mathfrak{g}$, it turns out that the action of $H$ on $\mathfrak{g}$ is trivial. Indeed, $\alpha_G$ is in this case a trivial vector bundle. This should not be any surprise, as $G\times_H \mathfrak{h}$ is well known to be an inverse for the tangent bundle $T(G/H)$ (which equals $G\times_H \mathfrak{p}$, where $\mathfrak{p}$ denotes a complement of the Lie algebra of $H$ in $G$ endowed with the isotropy action of $H$, see e.g. Michor's book \cite[Section~18.16]{Michor}).
\end{remark}


\section{\texorpdfstring{The Hirzebruch $\chi_y$ genus and Theorem 1.3}{The Hirzebruch Xy genus and Theorem 1.3}}\label{appendix}

The purpose of this section is to introduce the Hirzebruch $\chi_y$ genus in order to provide a proof of Theorem \ref{thm:Hirzebruch}, a result which appears in Hirzebruch's collected works \cite[page 777]{HirCW}. We begin with a brief overview of spin$^c$ structures, see {\color{black} Lawson and Michelsohn's book} \cite[Appendix D]{LM} for more details.

Recall that $Spin^c(n) = (Spin(n)\times U(1))/\langle(-1,-1)\rangle$, so there is a natural projection map $\pi : Spin^c(n) \to Spin(n)/\langle -1\rangle\times U(1)/\langle -1\rangle \cong SO(n)\times U(1)$. Let $P_{SO(n)}$ be an $SO(n)$-principal bundle over a topological space $B$. A \textit{spin$^c$ structure} on $P_{SO(n)}$ consists of a $U(1)$-principal bundle $P_{U(1)}$ and a $Spin^c(n)$-principal bundle $P_{Spin^c(n)}$ with a bundle map $P_{Spin^c(n)} \to P_{SO(n)}\times_B P_{U(1)}$ which is $Spin^c(n)$-equivariant where the action of $Spin^c(n)$ on $P_{SO(n)}\times_B P_{U(1)}$ is given by $g\cdot p := \pi(g)\cdot p$. We call the complex line bundle associated to $P_{U(1)}$ the \textit{canonical line bundle} of the spin$^c$ structure. 

\begin{example}\label{Jspin^c}
The oriented orthonormal frame bundle of an $n$-dimensional Hermitian vector bundle (with respect to the induced orientation and Riemannian metric) admits a natural spin$^c$ structure because the map $U(n) \to SO(2n)\times U(1)$, $A \mapsto (A, \det(A))$ lifts through $\pi$, and the canonical line bundle is the determinant line bundle.

There is another natural spin$^c$ structure in this situation. The map $U(n) \to SO(2n)\times U(1)$, $A \mapsto (A, \det(A)^{-1})$ also lifts through $\pi$, and the canonical line bundle of the corresponding spin$^c$ structure is the dual of the determinant line bundle. In particular, for the tangent bundle of a Hermitian manifold $X$, the canonical line bundle of this spin$^c$ structure is $K_X$, the canonical bundle. While this spin$^c$ structure would provide more consistency of the term ‘canonical’, the spin$^c$ structure mentioned initially is more prevalent in the literature.
\end{example} 

Given a spin$^c$ structure on $P_{SO(n)}$, we can form the associated complex spin$^c$ vector bundle $\mathbb{S} = P_{Spin^c(n)}\times_{\Delta}V$ where $\Delta : Spin^c(n) \to GL(V)$ is the restriction of a non-trivial irreducible complex representation of the complexified Clifford algebra. For every $n$, the representation $\Delta$ is unique, so the bundle $\mathbb{S}$ is well-defined. For $n$ even, the representation $\Delta$ is reducible, and we obtain a splitting $\mathbb{S} = \mathbb{S}^+\oplus\mathbb{S}^-$ with $\rank_{\mathbb{C}}\mathbb{S}^+ = \rank_{\mathbb{C}}\mathbb{S}^-$.

\begin{example}\label{Jspinbundles}
Equip the oriented orthonormal frame bundle of a Hermitian vector bundle $V$ with the spin$^c$ structure from Example \ref{Jspin^c}. Then $\mathbb{S} = \bigwedge^*V$ which decomposes into $\mathbb{S}^+ = \bigwedge^{\text{even}}V$ and $\mathbb{S}^- = \bigwedge^{\text{odd}}V$; see {\color{black} Morgan's book} \cite[Cor 3.4.5]{M} {\color{black} for more details}.
\end{example}

A spin$^c$ structure on an oriented Riemannian manifold $(M, g)$ is a spin$^c$ structure on the oriented orthonormal frame bundle -- such a structure exists if and only if the third integral Stiefel-Whitney class $W_3(M)$ vanishes. Suppose now that a spin$^c$ structure has been fixed. Using the fact that $\mathbb{S}$ is not only a complex vector bundle, but a bundle of Clifford modules, one can construct a self-adjoint elliptic differential operator $D: \Gamma(\mathbb{S}) \to \Gamma(\mathbb{S})$ called the spin$^c$ Dirac operator. When $n$ is even, we have $D(\Gamma(\mathbb{S}^{\pm})) \subseteq \Gamma(\mathbb{S}^{\mp})$ so there is an induced elliptic operator $D|_{\Gamma(\mathbb{S}^+)} : \Gamma(\mathbb{S}^+) \to \Gamma(\mathbb{S}^-)$ which we denote $\slashed{\partial}^c$. The index of $\slashed{\partial}^c$ is given by
$$\operatorname{ind}(\slashed{\partial}^c) = \int_M\exp(c_1(L)/2)\hat{A}(TM),$$ 
where $L$ is the canonical line bundle of the spin$^c$ structure. Note that for the spin$^c$ Dirac operator $D$ we have $\operatorname{ind}(D) = 0$ as $D$ is self-adjoint.

If $E \to M$ is a Hermitian vector bundle, then there is a twisted spin$^c$ Dirac operator $D_E$ which gives rise to the operator $\slashed{\partial}^c_E : \Gamma(\mathbb{S}^+\otimes E) \to \Gamma(\mathbb{S}^-\otimes E)$. The index of $\slashed{\partial}^c_E$ is 
$$\operatorname{ind}(\slashed{\partial}_E^c) = \int_M\exp(c_1(L)/2)\operatorname{ch}(E)\hat{A}(TM),$$
{\color{black} as shown by Schwarzenberger in an appendix to a book of Hirzebruch \cite[Theorem 26.1.1]{Hir}}. 

Suppose now that $\dim M = 2n$ and $M$ admits almost complex structures. We fix an almost complex structure $J$ such that $g$ is Hermitian. By Example \ref{Jspin^c}, there is a spin${}^c$ structure on $(M, g)$ which has associated line bundle $L = \det_{\mathbb{C}}(TM)$. As $c_1(L) = c_1(M)$ and $\exp(c_1(M)/2)\hat{A}(TM) = \operatorname{Td}(TM)$, the index of the operator $\slashed{\partial}^c_E$ is given by
\[\operatorname{ind}(\slashed{\partial}_E^c) = \int_M\operatorname{ch}(E)\operatorname{Td}(TM).\tag{$\ast$}\]
Since $T^{1,0}M \cong (T^{0,1}M)^*$, we have $\mathbb{S}^+ \cong \bigwedge^{0,\text{even}}M$ and $\mathbb{S}^- \cong \bigwedge^{0,\text{odd}}M$ by Example \ref{Jspinbundles}. The Hermitian metric $g$ induces a Hermitian metric on $\bigwedge^{p,0}M$, so we can construct the operator $\slashed{\partial}_{\BigWedge^{p,0}M}^c$. Under the appropriate identifications, we can view $\slashed{\partial}_{\BigWedge^{p,0}M}^c$ as an operator $\Gamma(\bigwedge^{p,\text{even}}M) \to \Gamma(\bigwedge^{p,\text{odd}}M)$; for notational convenience, we will denote this operator by $\slashed{\partial}^c_p$. Setting $\chi^p(M) := \operatorname{ind}(\slashed{\partial}_p^c)$, the \textit{Hirzebruch $\chi_y$ genus} is defined to be
$$\chi_y(M) := \sum_{p=0}^n\chi^p(M)y^p.$$

Note that $\chi_0(M) = \chi^0(M)$ is precisely the Todd genus of $M$ by $(\ast)$.

Suppose now that $J$ is integrable, in which case $n = \dim_{\mathbb{C}}M$. Then, modulo order zero terms, {\color{black} Gauduchon \cite[Proposition 8]{Gau97} has shown that }we have $\slashed{\partial}^c = \sqrt{2}(\bar{\partial} + \bar{\partial}^*)$ . In addition, if $E$ is holomorphic, then modulo order zero terms $\slashed{\partial}^c_E = \sqrt{2}(\bar{\partial}_E + \bar{\partial}_E^*)$ and $(\ast)$ reduces to the Hirzebruch-Riemann-Roch theorem. In particular, as $\bigwedge^{p,0}M$ is holomorphic and $\slashed{\partial}^c_p : \Gamma(\bigwedge^{p, \text{even}}M) \to \Gamma(\bigwedge^{p, \text{odd}}M)$ is just $\sqrt{2}(\bar{\partial} + \bar{\partial}^*)$ to highest order, we have 
$$\chi^p(M) = \operatorname{ind}(\slashed{\partial}^c_p) = \operatorname{ind}\sqrt{2}(\bar{\partial} + \bar{\partial}^*) = \sum_{q=0}^n(-1)^qh^{p,q}(M) = \chi(M, \Omega^p),$$ 
where $h^{p,q}(M)$ denotes the $(p, q)$-th Hodge number of $M$, and $\chi(M, \Omega^p)$ denotes the holomorphic Euler characteristic of $\Omega^p$.

Returning to the general case, we will need three properties of the $\chi_y$ genus: 

{\bf Property 1.} $\chi_{-1}(M) = \chi(M)$.

{\bf Property 2.} $\chi_y(M) = (-y)^n\chi_{y^{-1}}(M)$, equivalently $\chi^p(M) = (-1)^n\chi^{n-p}(M)$.

{\bf Property 3.} $\chi_1(M) = \sigma(M)$.

In order to establish these, we need the following lemma:

\begin{lemma}
Let $x_i$ be the Chern roots of $TM$. Then
$$\chi_y(M) = \int_M\prod_{i=1}^n\frac{x_i(1+ye^{-x_i})}{1-e^{-x_i}}.$$
\end{lemma}
\begin{proof}
Without loss of generality, we can suppose that $TM$ is isomorphic to a direct sum of complex line bundles $\ell_1\oplus\dots\oplus\ell_n$ by the splitting principle, and hence $T^*M\cong\ell_1^*\oplus\dots\oplus\ell_n^*$. Defining $x_i = c_1(\ell_i)$, we have $-x_i = c_1(\ell_i^*)$. Note that 
$$\bigwedge\nolimits^{p,0}M = \bigwedge\nolimits^pT^*M = \bigwedge\nolimits^p(\ell_1^*\oplus\dots\oplus\ell_n^*) = s_p(\ell_1^*, \dots, \ell_n^*),$$ 
where $s_p$ is the $p$-th elementary symmetric polynomial (addition and multiplication correspond to direct sum and tensor product respectively). Therefore
$$\operatorname{ch}\left(\bigwedge\nolimits^{p,0}M\right) = \operatorname{ch}(s_p(\ell_1^*, \dots, \ell_n^*)) = s_p(\operatorname{ch}(\ell_1^*), \dots, \operatorname{ch}(\ell_n^*)) = s_p(e^{-x_1}, \dots, e^{-x_n}).$$

So we have
\begin{align*}
\chi_y(M) &= \sum_{p=0}^n\chi^p(M)y^p\\ 
&= \sum_{p=0}^n\operatorname{ind}(\slashed{\partial}^c_p)y^p\\ 
&= \sum_{p=0}^n\left(\int_M\operatorname{ch}\left(\bigwedge\nolimits^{p,0}M\right)\operatorname{Td}(M)\right)y^p\\
&= \int_M\left(\sum_{p=0}^n\operatorname{ch}\left(\bigwedge\nolimits^{p,0}M\right)y^p\right)\operatorname{Td}(M)\\
&= \int_M\left(\sum_{p=0}^n s_p(e^{-x_1}, \dots, e^{-x_n})y^p\right)\prod_{i=1}^n\frac{x_i}{1-e^{-x_i}}\\
&= \int_M\prod_{i=1}^n(1 + e^{-x_i}y)\prod_{i=1}^n\frac{x_i}{1-e^{-x_i}}\\
&= \int_M\prod_{i=1}^n\frac{x_i(1 + e^{-x_i}y)}{1-e^{-x_i}}.
\end{align*}
\end{proof}

We can now demonstrate the first property.

{\bf Property 1.} $\chi_{-1}(M) = \chi(M)$.
\begin{proof}
Setting $y = -1$, we have
$$\chi_{-1}(M) = \int_M\prod_{i=1}^n\frac{x_i(1 - e^{-x_i})}{1-e^{-x_i}} = \int_M\prod_{i=1}^n x_i = \int_M c_n(M) = \int_M e(M) = \chi(M).$$
\end{proof}

When $J$ is integrable, combining Property 1 with the identity $\chi^p(M) = \sum\limits_{q=0}^n(-1)^qh^{p,q}(M)$ we obtain the following result (which also follows from the existence of the Fr\"olicher spectral sequence):

\begin{corollary}
Let $M$ be an $n$-dimensional compact complex manifold. The Euler characteristic of $M$ is given by
$$\chi(M) = \sum_{p=0}^n\sum_{q=0}^n(-1)^{p+q}h^{p,q}(M).$$
\end{corollary}

For the remaining properties, we also need the following lemma which the first author learnt from Ping Li:

\begin{lemma}\label{tlemma}
Let $t$ be a parameter. Then 
$$\chi_y(M) = \int_M\prod_{i=1}^n\frac{x_i(1 + ye^{-tx_i})}{1 - e^{-tx_i}}.$$
\end{lemma}

\begin{proof}
The key is to note that $\chi_y(M)$ only depends on the degree $2n$ part of the integrand. As $\deg x_i = 2$, if we replace $x_i$ by $tx_i$, then we have
$$\int_M\prod_{i=1}^n\frac{tx_i(1 + ye^{-tx_i})}{1 - e^{-tx_i}} = t^n\int_M\prod_{i=1}^n \frac{x_i(1 + ye^{-x_i})}{1 - e^{-x_i}} = t^n\chi_y(M).$$
Dividing through by $t^n$, we arrive at the result.
\end{proof}

{\bf Property 2.} $\chi_y(M) = (-y)^n\chi_{y^{-1}}(M)$, equivalently $\chi^p(M) = (-1)^n\chi^{n-p}(M)$.
\begin{proof}
Using $t = -1$ in the expression for $\chi_y(M)$ from Lemma \ref{tlemma}, we have

\begin{align*}
\chi_y(M) &= \int_M\prod_{i=1}^n\frac{x_i(1+ye^{x_i})}{1-e^{x_i}}\\
&= \int_M\prod_{i=1}^n(-y)\frac{x_i(y^{-1} + e^{x_i})}{e^{x_i}-1}\\
&= (-y)^n\int_M\prod_{i=1}^n\frac{x_i(e^{x_i}+y^{-1})}{e^{x_i}-1}\\
&= (-y)^n\int_M\prod_{i=1}^n\frac{x_i(1+y^{-1}e^{-x_i})}{1-e^{-x_i}}\\
&= (-y)^n\chi_{y^{-1}}(M).
\end{align*}

Comparing coefficients of $y^p$ gives the equivalent statement.
\end{proof}

In the integrable case, the equality $\chi^p(M) = (-1)^n\chi^{n-p}(M)$ follows from Serre duality.

{\bf Property 3.} $\chi_1(M) = \sigma(M)$.
\begin{proof}
If $n$ is odd, then $\chi_1(M) = 0$ by Property 2, and $\sigma(M) = 0$ by definition. 

If $n$ is even, let $n = 2m$. Setting $y = 1$ gives
$$\chi_1(M) = \int_M\prod_{i=1}^n \frac{x_i(1 + e^{-x_i})}{1-e^{-x_i}} = \int_M\prod_{i=1}^n \frac{x_i(1 + e^{-2x_i})}{1-e^{-2x_i}}= \int_M\prod_{i=1}^n\frac{x_i}{\tanh(x_i)},$$
where the second equality uses Lemma \ref{tlemma} with $t = 2$. Recall that the power series which generates the $L$ genus in terms of Pontryagin classes is
$$Q(x) = \frac{\sqrt{x}}{\tanh(\sqrt{x})}.$$
By \cite[Lemma 1.3.1]{Hir}, the corresponding power series which generates the $L$ genus in terms of Chern classes is
$$\widetilde{Q}(x) = \frac{x}{\tanh(x)}.$$
So the above computation shows that 
$$\chi_1(M) = \int_M\prod_{i=1}^n\frac{x_i}{\tanh(x_i)} = \int_M L_m(p_1, \dots, p_m) = \sigma(M).$$
\end{proof}

In the integrable case, we immediately obtain the following corollary.

\begin{corollary}
Let $M$ be an $n$-dimensional compact complex manifold. The signature of $M$ is given by
$$\sigma(M) = \sum_{p=0}^n\sum_{q=0}^n(-1)^qh^{p,q}(M).$$
\end{corollary}

\begin{remark}If $M$ admits a K\"ahler metric (or more generally, has the $\partial\bar{\partial}$-property), then $h^{p,q}(M) = h^{q,p}(M)$. It follows that we can replace $(-1)^q$ by $(-1)^p$ in the above summation, in which case the identity follows from the Hard Lefschetz theorem \cite[Corollary 3.3.18]{Huy}. In general, the modified summation does not compute the signature as can easily by checked for a Hopf surface.

\end{remark}

Now that we have all three properties of $\chi_y(M)$ at our disposal, we can give a proof of Theorem \ref{thm:Hirzebruch}{\color{black}, which we recall for the readers' convenience.}

\begin{theorem}
Suppose $M$ is a closed $4m$-dimensional manifold which admits an almost complex structure, then $\chi(M) \equiv (-1)^m\sigma(M) \bmod 4$.
\end{theorem}
\begin{proof}
We separate the proof into two cases based on the parity of $m$.

If $m = 2k$ is even, then $n = 2m = 4k$ so
\begin{align*}
\chi(M) &= \chi_{-1}(M)\\
&=\sum_{p=0}^{4k}(-1)^p\chi^p(M)\\
&= \sum_{p=0}^{4k}\chi^p(M) - 2\sum_{p=0}^{2k-1}\chi^{2p+1}(M)\\
&= \chi_1(M) - 2\left[\sum_{p=0}^{k-1}\chi^{2p+1}(M) + \sum_{p = k}^{2k-1}\chi^{2p+1}(M)\right]\\
&= \sigma(M) - 2\left[\sum_{p=0}^{k-1}\chi^{2p+1}(M) + \sum_{p = k}^{2k - 1}(-1)^{4k}\chi^{4k-(2p+1)}(M)\right]\\
&= \sigma(M) - 2\left[\sum_{p=0}^{k-1}\chi^{2p+1}(M) + \sum_{p=k}^{2k-1}\chi^{2(2k-1 -p)+1}(M)\right]\\
&= \sigma(M) - 2\left[\sum_{p=0}^{k-1}\chi^{2p+1}(M) + \sum_{p=0}^{k-1}\chi^{2p+1}(M)\right]\\
&= \sigma(M) - 4\sum_{p=0}^{k-1}\chi^{2p+1}(M).
\end{align*}
Therefore $\chi(M) \equiv \sigma(M) \bmod 4$.

If $m = 2k + 1$ is odd, then $n = 2m = 4k + 2$ so
\begin{align*}
\chi(M) &= \chi_{-1}(M)\\
&= \sum_{p=0}^{4k + 2}(-1)^p\chi^p(M)\\
&= -\sum_{p=0}^{4k+2}\chi^p(M) + 2\sum_{p=0}^{2k+1}\chi^{2p}(M)\\
&= -\chi_1(M) + 2\left[\sum_{p=0}^k\chi^{2p}(M) + \sum_{p=k+1}^{2k+1}\chi^{2p}(M)\right]\\
&= -\sigma(M) + 2\left[\sum_{p=0}^k\chi^{2p}(M) + \sum_{p=k+1}^{2k+1}(-1)^{4k+2}\chi^{4k+2-2p}(M)\right]\\
&= -\sigma(M) + 2\left[\sum_{p=0}^k\chi^{2p}(M) + \sum_{p=k+1}^{2k+1}\chi^{2(2k+1-p)}(M)\right]\\
&= -\sigma(M) + 2\left[\sum_{p=0}^k\chi^{2p}(M) + \sum_{p=0}^k\chi^{2p}(M)\right]\\
&= -\sigma(M) + 4\sum_{p=0}^k\chi^{2p}(M).
\end{align*}
Therefore $\chi(M) \equiv -\sigma(M) \bmod 4$.
\end{proof}


\section{Examples using representation theory}\label{S:examples_spin}


The goal of this section is to find examples of biquotients $G\bq H$ for which $T(G\bq H)\notin \Bv(G\bq H)$. The examples of this section are of the form $H_1\backslash G/H_2$ for certain subgroup inclusions $H_i\subseteq G$, see Subsection~\ref{SS:definitions}. We first fix some notation for the rest of the section.

We let $\Delta S^1\subseteq Spin(2n+2)$ be the image of the lift of the composition $S^1\rightarrow SO(2)\rightarrow SO(2n+2)$, where the first map is a double cover, and the second map is the block diagonal embedding $B\mapsto \diag(B,B, ...,B)$.  

We denote by $\Delta SU(2)\subseteq Spin(4n+4)$ the image of the lift of the composition $SU(2)\rightarrow SO(4)\rightarrow SO(4n+4)$. Here the first map $SU(2)\rightarrow SO(4)$ is the natural inclusion, i.e. the one induced by regarding $\CC^2$ as a real vector space isomorphic to $\RR^4$. The second map $SO(4)\rightarrow SO(4n+4)$ is the block diagonal embedding $A\mapsto \diag(A,A,...,A)$.

Finally, a spin subgroup $Spin(k)$ of a spin group $Spin(k+m)$ will always be understood to be embedded via the lift of the standard block inclusion $A\mapsto \diag(A,1,...,1)$.

We shall begin with the examples of Theorem~\ref{thm:presentations_of_HP}. As shown by Eschenburg \cite[Table 101]{Es84}, the previous inclusions induce well-defined biquotients $\Delta S^1\backslash Spin(2n+2)/Spin(2n+1)$ and $\Delta SU(2)\backslash Spin(4n+4)/Spin(4n+3)$ that are diffeomorphic to $\mathbb{C}P^n$ and $\HH P^n$ respectively. We note that these biquotients are obviously reduced, being quotients of simple groups.  Let $\mathbb{K} \in \{\mathbb{C}, \mathbb{H}\}$.  Except for the case of $\mathbb{C}P^1$,  we will show that $T\mathbb{K}P^n$ is not a biquotient bundle for the above presentations. A more precise version of Theorem~\ref{thm:presentations_of_HP} is given by the following result. 


\begin{theorem}\label{thm:pres_projective_spaces}
Consider the reduced presentations $\mathbb{C}P^n \cong \Delta S^1\backslash Spin(2n+2)/Spin(2n+1)$ with $n\geq 2$ and $\mathbb{H}P^n\cong \Delta SU(2)\backslash Spin(4n+4)/Spin(4n+3)$ with $n\geq 1$. Then $T\mathbb{K}P^n$ is not a biquotient bundle with respect to these presentations.
\end{theorem}

Before giving its proof, we provide two lemmas that will be useful throughout this section.

First, we shall use the following well-known facts about the representations of the group $Spin(2n+1)$. Following {\color{black} Bröcker and tom Dieck's book} \cite{BtD85}, let $\Lambda=\Lambda^1$ be the standard representation, let $\Lambda^i$ the $i$-th exterior product, and let $\Delta$ be the half-spin representation.

\begin{lemma}\label{LEM:rep_of_spin}
For $n \geq 1$, the non-trivial real irreducible representation of $Spin(2n+1)$ of smallest dimension is the standard one $\Lambda^1$, which is of dimension $2n+1$.  {\color{black} For $n\geq 3$ odd, $Spin(4n-1)$ has no other real irreducible representations of dimension at most $8n-4$.} 
\end{lemma}

\begin{proof}
The fundamental irreducible representations are exactly $\Lambda^i$ with $1\leq i \leq n$ and $\Delta$.  It follows from the Weyl dimension formula that the smallest dimensional representations of a simply connected compact Lie group are found among its fundamental representations.


Since $\dim \Lambda^i={2n+1 \choose i}$, it follows that $\dim \Lambda^{i}\leq \dim \Lambda^{i+1}$ for all $i$. Thus, in order to find the smallest and second smallest representations we are left to decide between $\Lambda^1$, $\Lambda^2$ and $\Delta$.

From \cite[Lemma~6.3, p.280]{BtD85}, we know that the representation $\Delta$ of $Spin(2n+1)$ has complex dimension $2^n$.  {\color{black} For $n\geq 3$, it follows easily that $\dim \Delta > \dim \Lambda$.  For $n\leq 2$, the spin representation is not real \cite[Proposition ~6.19, p.290]{BtD85}.  Thus, $\Lambda$ is always the smallest non-trivial real representation of $SO(2n+1)$.} This proves the first part of the statement.


{\color{black} Next, we verify that the second smallest real irreducible representation of $Spin(4n-1)$ has dimension larger than $8n-4$.  For $\Lambda^2$, we have $$\dim \Lambda^2 = {4n-1 \choose 2} = 8n^2-6n+1 > 8n-4$$ for any $n\geq 3$.  For $\Delta$, we have $\dim \Delta = 2^{2n-1} > 8n-4$ for any $n\geq 3.$} 
\end{proof}



Second, we will use the following bundle-theoretic result.

\begin{lemma}\label{LEM:euler_class_odd}
Let $E$ be an orientable vector bundle over a closed manifold $M$ such that all cohomology groups $H^i(M)$ are torsion-free. If $E$ splits as $E=E_1\oplus E_2$  with the subbundle $E_1$ of odd rank, then the Euler class $e(E)$ vanishes.
\end{lemma}

\begin{proof}
If $E = E_1\oplus E_2$ with $E_1$ odd rank, then the map which takes $(e_1,e_2)\mapsto (-e_1,e_2)$ is a bundle isomorphism which reverses orientation. The Euler class satisfies $e(E) = e(-E) = -e(E)$.  Altogether we get that $2e(E) = 0$. The cohomological assumption on $M$ implies that $e(E)=0$.
\end{proof}

We are ready to give the proof of Theorem~\ref{thm:pres_projective_spaces}.  Recall that a manifold $M$ is said to be almost complex if its tangent bundle $TM$ is isomorphic to the realification of a complex vector bundle.

\begin{proof}[Proof of Theorem~\ref{thm:pres_projective_spaces}] Assume for a contradiction that we can find a representation $V$ of $H$ for which $T\mathbb{K}P^n\cong G\times_H V$. Here $$(G,H):= \begin{cases} (Spin(2n+2), Spin(2n+1) \times S^1) & \text{ if }\mathbb{K}=\CC,\\ (Spin(4n+4), Spin(4n+3)\times SU(2)) & \text{ if } \mathbb{K}=\HH.\end{cases}$$


From the first part of Lemma~\ref{LEM:rep_of_spin} we know that the (real) dimension of the smallest non-trivial representation of the $Spin$ factor of $H$ is $2n+1$ for $\mathbb{K}=\CC$ (resp. $4n+3$ for $\mathbb{K}=\HH$), which is strictly larger than $\dim V=\dim \CC P^n=2n$ (resp. $\dim V=\dim \HH P^n = 4n$). In particular, in all cases, the $Spin$ factor of $H$ must act trivially on $V$.  So, we can consider $V$ as a representation of $S^1$ or $SU(2)$, respectively.  


For the case $\mathbb{K} = \mathbb{C}$, the representation $V$ must split as a sum of line bundles, which implies a corresponding splitting of $T\mathbb{C}P^n$. However, there is no such such splitting, as shown by Glover, Homer and Stong \cite[Theorem~1.1.(ii)]{GHS}.

We now turn to the case $\mathbb{K} = \mathbb{H}$. Since the Euler class $e(T\mathbb{H}P^n)$ is non-zero, $T\mathbb{H}P^n$ cannot contain any odd rank subbundles by Lemma~\ref{LEM:euler_class_odd}. Thus, all irreducible subrepresentations of $V$ are even dimensional. Since even dimensional irreducible representation of $SU(2)$ are complex, it follows that $V$ is a complex vector space, so $\mathbb{H}P^n$ is almost complex.  However, Massey \cite{Massey} has shown that $\mathbb{H}P^n$ is not almost complex for any $n\geq 1$.
\end{proof}


Following a similar line of argument, we now give examples for which the tangent bundle is not a biquotient vector bundle for any presentation. 

We denote by $\Delta SU(2)\subseteq Spin(4n+1)$ the image of the lift of the composition $SU(2)\rightarrow SO(4)\rightarrow SO(4n+1)$. Here the first map $SU(2)\rightarrow SO(4)$ is the same as above, while the second map $SO(4)\rightarrow SO(4n+1)$ is the block diagonal embedding $A\mapsto \diag(A,A,...,A,1)$.   Again by work of Eschenburg \cite[Table 101]{Es84}, these inclusions induce biquotients:
$$M^{8n-4} := \Delta SU(2)\backslash Spin(4n+1)/Spin(4n-1).$$

{\color{black} These spaces with $n$ odd are the $(16n+4)$-dimensional examples of Theorem \ref{thm:first_theorem}.  That is, we will show:}

\begin{theorem}\label{thm:examples}
Let $n\geq 3$ odd.  {\color{black}Then $TM\notin\Bv(G\bq H)$ for any presentation of $M$ as a biquotient.} 
\end{theorem}

{\color{black} Towards proving Theorem \ref{thm:examples}, we first observe that it  follows from }Propositions \ref{prop:cover} and \ref{prop:reduced}, that if $TM$ {\color{black} $\in \Bv(G\bq H)$ for some biquotient presentation of $M$, then $TM$ is also a biquotient vector bundle with respect to a }reduced presentation with $G$ a product of a simply connected compact Lie group with a torus.  Further, from \cite[Lemma~3.3]{To02}, the fact that it is reduced implies that the torus factor is trivial. Kapovitch and Ziller showed in \cite[Section~2]{KZ04} that the only reduced presentation $G\bq H$ of $M$ with $G$ simply connected is the one above. We note that Kapovitch and Ziller do not use the notion of ``reduced biquotient'', but their classification of presentations $G\bq H$ with $G$ simple is equivalent to classifying all such reduced presentations. {\color{black} Thus, to prove Theorem \ref{thm:examples}, it is enough to show that $TM\notin \Bv(\Delta SU(2)\backslash Spin(4n+1)/Spin(4n-1)).$}  

\begin{proof} As shown in \cite{KZ04}, $M$ has the same integral cohomology groups of $\mathbb{H}P^{2n-1}$. Since $(8n-4)/2 = 4n-2 \equiv 2\pmod{4}$, it follows that $H^{4n-2}(M)=0$ and hence the signature $\sigma(M)$ vanishes for trivial reasons.  On the other hand, $\chi(M) = 2n$.  Thus, if $n$ is odd, then $\chi(M)\equiv 2\pmod{4}$. In particular, Theorem  \ref{thm:Hirzebruch} implies that $M$ cannot admit an almost complex structure for $n$ odd.

{\color{black} Set $G  = Spin(4n+1)$ and $H = SU(2)\times Spin(4n-1)$ and} assume for a contradiction that {\color{black} there is a representation $V$ of $H = SU(2)\times Spin(4n-1)$ for which  $TM\cong G\times_H V\in \Bv(G\bq H).$} 

From the second part of Lemma~\ref{LEM:rep_of_spin}, we find that the only non-trivial irreducible representations of $Spin(4n-1)$ of dimension at most $8n-4=\dim M^{8n-4}$ is the standard representation $\Lambda^1$ (of real dimension $4n-1$).


Let $W_k$ denote the unique irreducible representation of $SU(2)$ of complex dimension $k$.  Then the irreducible representation $W_k\otimes \Lambda^1$ has dimension $k(4n-1)$.  Hence, if $W_k\otimes \Lambda^1$ is a subrepresentation of $V$, then we must have $k(4n-1) \leq 8n-4$ and hence $k=1$. But if $k=1$, then $TM$ has an odd rank subbundle induced by $W_1\otimes \Lambda^1$. This is a contradiction to Lemma~\ref{LEM:euler_class_odd}, since $\chi(M) = 2n$ and hence $e(TM)\neq 0$.

Thus, $V$ must be a representation of $SU(2)$ only. Again, because $TM$ cannot have any odd rank subbundles, all irreducible subrepresentations of $V$ are even dimensional. Since even dimensional irreducible representations of $SU(2)$ are complex, it follows that $V$ is a complex vector space, so $TM$ is almost complex. This is a contradiction.
\end{proof}


\section{Examples using characteristic classes}\label{S:examples_torus}

In this section, we complete the proof of Theorem \ref{thm:first_theorem} by constructing examples in dimensions $4$ and $6$, including infinitely many homotopy types in dimension $6$.

These biquotients will all be of the form $G\bq T^k$ for some biquotient action of a torus $T^k$ on a semi-simple compact Lie group $G$. The set of biquotient vector bundles over $G\bq T^k$ can be characterized as follows (see~\cite[Section~3]{DG21}).


\begin{proposition}\label{prop:biq_bundles}
A real vector bundle over $G\bq T$ is a biquotient bundle if and only if it is the realification of a sum of complex line bundles, possibly summed with a trivial real vector bundle. 
\end{proposition}

In this section we will only consider $G\bq T^k$ where $T^k$ is of maximal rank. It follows that its dimension is even, say $2m$, and its Euler characteristic is positive. Moreover, we will consider only biquotient bundles of real rank $2m$. In other words, we only consider (realifications of) Whitney sums $L=\oplus L_i$ of $m$ complex line bundles $L_i$, i.e.~$1\leq i\leq m$. The relevant characteristic classes of the realification $rL$ will be the Pontryagin class $p_1$ and the Euler class $e$, which can be computed as:
\begin{align*}
p_1( rL) & = -c_2(L\oplus \bar L) = \sum_{i=1}^m c_1^2(L_i) && \in H^4(G\bq T^k) \\
e( rL ) & = c_{m}(L)=\prod_{i=1}^{m}c_1(L_i) && \in H^{2m}(G\bq T^k),
\end{align*}
where $\bar L$ denotes the conjugate bundle of $L$ and $c_i$ the $i$-th Chern class.

In the following subsections we will consider biquotients $N:=G\bq T^k$ and prove that $T N$ cannot be isomorphic to $rL$. This will be done in three steps:
\begin{itemize}
\item we compute/collect/recall the necessary topological properties of $N$, including the values of $p_1(N)$ and $\pm e(N)$,
\item we compute $p_1(rL)$ and $e(rL)$ for any $L=\oplus_{i=1}^m L_i$ with $2m=\dim N$,
\item we convert the cohomological equations $p_1(rL) = p_1(TN)$ and $e(rL) = \pm e(TN)$ into numerical equations and show that they have no common solution.
\end{itemize}
We remark that any homogeneous space $G/T^k$ with $T^k$ a maximal torus in $G$ is stably parallelizable and hence $p_1=0$, a fact which no longer holds in general for biquotients $G\bq T^k$.

\subsection{\texorpdfstring{The tangent bundle of $\mathbb{C}P^2\# \mathbb{C}P^2$}{The tangent bundle of CP2 \# CP2}} The manifold $N:=\mathbb{C}P^2\# \mathbb{C}P^2$ is a biquotient of the form $(S^3)^2\bq T^2$ \cite[p.~404]{To02}. Its cohomology ring is given by $H^\ast(N) \cong \mathbb{Z}[u,v]/\langle u^2 - v^2, uv\rangle$ with both $|u|=|v| = 2$. Using the Hirzebruch signature theorem, one finds that $p_1(N)=6u^2$. Since $\chi(N) = 4$, it follows that the Euler class is $e(N) = \pm 4 u^2$.

Assume for a contradiction that $TN \cong r L =r (L_1\oplus L_2)$, where $L_i$ is determined by $c_1(L_i)=a_i u + b_i v$. Then one easily computes:
$$
p_1(r L)=(a_1^2 + b_1^2 + a_2^2 + b_2^2)u^2 \text{ and } e(r L)=(a_1 a_2 + b_1 b_2)u^2.
$$
Since $TN \cong r L$, both $TN$ and $rL$ must have the same first Pontryagin class and the same Euler class and hence we have the two equations
\begin{align*}
6 &= a_1^2 + b_1^2 + a_2^2 + b_2^2\\
\pm 4 & = a_1 a_2 + b_1 b_2.
\end{align*}
From the first equation it follows that, up to permutations and signs, the only integer solution of this equation is $(a_1,b_1,a_2,b_2) = (2,1,1,0)$.  Assume without loss of generality that $b_2 = 0$. Hence the second equation simplifies to $\pm 4 = a_1 a_2$ with $(a_1,a_2)$ given up to permutation and signs as $(2,1)$ or $(1,1)$. In either case, $a_1a_2\neq 4$, so $TN$ is not a sum of line bundles.

Alternatively, one can argue that $TN$ is not a biquotient bundle as follows. If $TN$ was a biquotient bundle, then in particular it would be the realification of a complex vector bundle (see Proposition~\ref{prop:biq_bundles}). In other words, $\mathbb{C}P^2\# \mathbb{C}P^2$ would be an almost complex manifold, which leads to a contradiction as follows. The signature of $\mathbb{C}P^2\# \mathbb{C}P^2$ equals $2$ while its Euler characteristic equals $4$, hence Hirzebruch's Theorem~\ref{thm:Hirzebruch} tells us that $\mathbb{C}P^2\# \mathbb{C}P^2$ cannot admit an almost complex structure.
 
\subsection{\texorpdfstring{The tangent bundle of $SU(3)\bq T^2$}{The tangent bundle of SU(3) // T2}}

Here we consider the inhomogeneous biquotient $N:=SU(3)\bq T^2$, which was discovered by Eschenburg \cite[Section 42]{Es84}. Its cohomology ring can be described as $H^\ast(N,\mathbb{Z})\cong\mathbb{Z}[x,y]/I$ where $I$ is the ideal $I = \langle x^3, y^2 + xy - x^2\rangle$ and with $|x| = |y| = 2$. In particular, $H^4(N,\mathbb{Z})\cong \mathbb{Z}^2$ has generators $x^2$ and $xy$, and $H^6(N,\mathbb{Z})$ has generator $x^2 y$ with $y^3 = 2x^2y$ and $xy^2 = -x^2y$. Further, $p_1(N) = 8x^2$ and $e(N) = \pm 6x^2 y$. All this information can be extracted e.g. {\color{black} from the work of Escher and Ziller \cite[Section~4 and Proposition~5.14]{EZ14}}, where $SU(3)\bq T^2$ corresponds to $N_{-1}$ in their notation. 

Assume for a contradiction that $TN \cong r L =r (L_1\oplus L_2\oplus L_3)$, where $L_i$ is determined by $c_1(L_i) = a_i x + b_i y$. A simple calculation shows that 
\begin{align*}
p_1(rL) &= \sum (a_i^2 + b_i^2)x^2 + \sum (2a_i b_i - b_i^2)xy \text{ and}  \\
e(rL) &=(a_1a_2b_3 + a_1b_2a_3 + b_1a_2a_3 - a_1 b_2 b_3 - b_1a_2 b_3 - b_1b_2a_3 + 2b_1 b_2 b_3) x^2y.
\end{align*}
Since $p_1(rL) = \sum c_1(L_i)^2$ is a sum of squares and $e(rL)$ is only defined up to sign, we may replace any $L_i$ with $\bar{L_i}$ because $c_1(\bar{L_i}) = -c_1(L_i)$.  Thus, we may assume without loss of generality that $b_i\geq 0$ for all $i$. Since $TN\cong rL$ we obtain the numerical equations
\begin{align}
8 &= \sum (a_i^2 + b_i^2) \label{EQ:p_1_SU} \\ 
0 & = \sum (2a_i b_i - b_i^2) \label{EQ:p_11_SU} \\
\pm 6 &= a_1a_2b_3 + a_1b_2a_3 + b_1a_2a_3 - a_1 b_2 b_3 - b_1a_2 b_3 - b_1b_2a_3 + 2b_1 b_2 b_3 \label{EQ:e_SU}
\end{align}
with $b_i\geq 0$ for all $i$.

From \eqref{EQ:p_1_SU} it is easy to see that, up to rearranging the order of the $a_i$ and $b_i$, the only possibilities are a) two of $\{a_i,b_j\}$ are $\pm 2$ with the others all zero, or b)  one of $\{a_i,b_j\}$ is $\pm 2$, four are $\pm 1$, and one is zero.

For a), only two of $\{a_i, b_j\}$ are non-zero.  Looking at \eqref{EQ:e_SU}, all terms in the right hand side involve three distinct elements of $\{a_i,b_j\}$, so all terms are zero and yield the contradiction $\pm 6=0$.

For b), we break into subcases depending on whether some $b_j = 0$ or some $a_i = 0$.  So, first assume some $b_j = 0$, say, $b_1 = 0$.  Then, upon substituting this into \eqref{EQ:e_SU}, we find $\pm 6 = a_1a_2 b_3 + a_1b_2a_3 -a_1b_2b_3$. If, say, $b_2 = \pm 2$, then all $a_i$ as well as $b_3$ must be odd.  Then $a_1a_2 b_3 + a_1b_2a_3 -a_1b_2b_3$ is odd, so not equal to $\pm 6.$  Thus, we must have $b_2 = b_3 = 1$. Then \eqref{EQ:p_11_SU} simplifies to $1 = a_2 +a_3$.  {\color{black} Thus \eqref{EQ:e_SU} becomes 
$$ 
\pm 6 = a_1 a_2 + a_1 a_3 - a_1 = a_1(a_2+a_3-1)= a_1(1-1)=0,$$ 
which is a contradiction.} 

This concludes the subcase with some $b_j = 0$. 

It remains to rule out the subcase where some $a_i = 0$, say $a_1 = 0$. Rewrite \eqref{EQ:p_11_SU} as $\sum b_i^2 = 2 \sum a_i b_i$.  If all $b_i$ are odd, then the left side is odd, while the right side is even. Thus at least one $b_i$ must be even.  If $b_1 = 2$, then $\sum b_i^2 = 2 \sum a_i b_i$ gives $6 = 2 (0\cdot b_1  \pm 1  \pm 1)$, which has no solution. Thus, without loss of generality, we must have $b_2 = 2$. In this case $\sum b_i^2 = 2 \sum a_i b_i$ gives $6 = 2(2a_2 + a_3)$, hence $a_2 = a_3 = 1$ and consequently $b_1=b_3=1$. Substituting all this into \eqref{EQ:e_SU}, we find the contradiction
$$\pm 6 =b_1(a_2a_3 - a_2b_3-b_2a_3+2b_2b_3) = 1-1-2+4.$$


\subsection{\texorpdfstring{The tangent bundle of the $R(p)$ biquotients}{The tangent bundle of the R(p) biquotients}}

For each integer $p$ we define a biquotient action of $T^3$ on $(S^3)^3$ as follows.  Viewing $S^3\subseteq \mathbb{C}^2$ and denoting the $i$-th coordinate of $(S^3)^3$ by $(a_i,b_i)$, the action is defined by the rule: 
$$(w_1,w_2,w_3)\ast(a_1,b_1,a_2,b_2,a_3,b_3) =(w_1a_1,w_1w_2^2b_2,w_2a_2,w_1w_2b_2,w_3a_3,w_1^pw_3b_3).$$
As shown by the second author in \cite[Proposition~4.22]{De17} (see also \cite[Proposition~3.5]{DG21}), the above action is free. For each $p\in\ZZ$, we denote the corresponding biquotient by $R(p)$ (note that it is denoted by $R(p,0)$ in \cite[Proposition~4.23]{De17} and by $R(A)$ in \cite[Section~3]{DG21}, where $A$ equals the matrix $A= \begin{bmatrix} 1 &2 & 0 \\ 1 & 1 & 0 \\ p & 0 & 1 \end{bmatrix}.$)

In \cite[Proposition~3.8]{DG21}, we use the cohomology ring of $H^\ast(R(p))$ to show that as $p$ varies over odd primes, infinitely many homotopy types of biquotients arise.  In the present case, the biquotients of Theorem \ref{thm:first_theorem} will be of the form $R(2q)$, so this result does not directly apply.  However, restricting to the case where $q$ is an odd prime, the proof of \cite[Proposition~3.8]{DG21} carries over verbatim until the very end, when the three equations should be reduced mod $q$ instead of mod $p$ to obtain the same contradiction.  Thus, there are infinitely many homotopy types among these examples.

The following properties can be found in \cite[Section~2.4 and Proposition~4.26]{De17} (see also \cite[Proposition~3.6]{DG21}). The cohomology ring of $R(p)$ is given by $H^*(R(p)) \cong \ZZ [ u_1,u_2,u_3]/  I$ where all $|u_i| = 2$ and $I$ is the ideal $\langle u_1^2 + 2u_1 u_2,  u_2^2 + u_1 u_2, u_i^2 + p u_1 u_3\rangle$. In particular, we observe that $u_2^2 = (u_1 + u_2)^2$ and $u_2 (u_1 + u_2)  =0$ for any $p$. The Pontryagin class of $R(p)$ is given by 
$p_1(R(p))=(-6-2p^2)u_1u_2$. A fundamental class is $u_1 u_2 u_3$ and the Euler class is $e(R(p))=\pm 8 u_1 u_2 u_3$. There are relations in degrees $4$ and $6$ as follows:
\begin{align*}
u_1^2 &=-2u_1u_2 && u_1^3 =u_2^3 =u_1^2u_2 =u_1u_2^2 = 0 \\
u_2^2 &=-u_1u_2 && u_3^2 =-2p^2u_1u_2u_3\\
u_3^2 &= -pu_1u_3 && u_1^2u_3 = -2u_1u_2u_3 \\
&&& u_1u_3^2 =2pu_1u_2u_3\\
&&& u_2^2 u_3 =-u_1u_2u_3\\
&&& u_2 u_3^2 =-pu_1u_2u_3.
\end{align*}

From now on, suppose $p$ is even and $p\geq4$, and write $p=2q$ so that $q\geq 2$. In order to simplify the computations we consider an alternative description of $H^*(R(p))$ determined by the following basis.
\begin{align*}
v_1 &:= u_1 + u_2 \\
v_2 &:= u_2 \\
v_3 &:= u_3 +q(v_1-v_2)=qu_1 + u_3.
\end{align*}
{\color{black} With respect to this new basis, } the Pontryagin class {\color{black} now takes the form } $p_1(R(p))=(6+2p^2)v_1^2=(6+8q^2)v_1^2$. A fundamental class is $v_1^2v_3$ so that $e(R(p))=\pm 8 v_1^2v_3$. The new relations in degree $4$ and $6$ are:
\begin{align*}
v_2^2 & = v_1^2 && v_1^3 =v_2^3=v_1^2v_2=v_1v_2^2=v_1v_2v_3=v_1v_3^2=v_1v_3^2=0\\
v_1v_2&=0  && v_3^3 =2q^2v_1^2v_3\\
v_3^2 &= 2q^2 v_1^2 && v_2^2v_3 = v_1^2 v_3.
\end{align*}

Let $L_i$ be the complex line bundle determined by $a_iv_1+b_iv_2+c_iv_3$ with $1\leq i\leq 3$ and define $L=L_1\oplus L_2\oplus L_3$. The relevant characteristic classes are
\begin{align*}
p_1(rL) =&  \sum(a_i^2 + b_i^2 + 2q^2c_i^2)v_1^2 + 2\sum (a_ic_i) v_1v_3 + 2\sum (b_ic_i)v_2v_3 \text{ and } \\
e(rL) =&  (a_1a_2c_3 + a_1c_2a_3 + c_1a_2a_3 + 2q^2 c_1c_2c_3 + b_1b_2c_3 + b_1c_2b_3 + c_1b_2b_3)v_1^2v_3.
\end{align*}
Assume $TR(p)\cong rL$, so that we arrive at the numerical equations
\begin{align}
\pm 8 &= a_1a_2c_3 + a_1c_2a_3 + c_1a_2a_3 + 2q^2 c_1c_2c_3 + b_1b_2c_3 + b_1c_2b_3 + c_1b_2b_3 \label{EQ:R_p_euler} \\
6+8q^2 &= \sum(a_i^2 + b_i^2 + 2q^2c_i^2)\label{EQ:R_p_pont1} \\
0 &= \sum a_ic_i\label{EQ:R_p_pont2}\\
0 &= \sum b_ic_i.\label{EQ:R_p_pont3}
\end{align}
We can rewrite \eqref{EQ:R_p_pont1} as
$$
2q^2(4- \sum c_i^2) = \sum(a_i^2 + b_i^2) - 6
$$
and since the right hand side is at least $ -6$ it follows that
$$
\sum c_i^2\leq 4 + \frac{3}{q^2}.
$$
Since $q\geq 2$ by assumption and the $c_i$ are integers we obtain that $\sum c_i^2\leq 4$. Hence the only possibilities for $(c_1,c_2,c_3)$ up to ordering and signs are 
$$
(0,0,0),\quad (2,0,0),\quad (1,1,0),\quad (1,1,1),\quad (1,0,0).
$$
The rest of the proof is dedicated to show that none of these possibilities can occur.

\begin{itemize}
\item[$(0,0,0)$:] This case cannot happen since \eqref{EQ:R_p_euler} yields the contradiction $\pm 8=0$.
\item[$(2,0,0)$:] Assume without loss of generality that $c_1=\pm 2$ and $c_2=c_3=0$. Then from \eqref{EQ:R_p_pont2} and \eqref{EQ:R_p_pont3} we get that $a_1=b_1=0$. Then \eqref{EQ:R_p_euler} and \eqref{EQ:R_p_pont1} become:
$$
4= |a_2a_3 + b_2b_3|\qquad 6=a_2^2+a_3^2+b_2^2+b_3^2.
$$
The last equation implies that the unique possibility for $(a_2,a_3,b_2,b_3)$ up to ordering and signs is $(2,1,1,0)$. None of the possible choices \textcolor{black}{satisfy} the equation $4= |a_2a_3 + b_2b_3|$.
\item[$(1,1,0)$:] Assume without loss of generality that $c_1,c_2\in\{\pm 1\}$ and $c_3=0$. Then \eqref{EQ:R_p_euler} becomes:
$$
\pm 8 =  a_3(a_1c_2+ c_1a_2) + b_3(b_1c_2 + c_1b_2).
$$
We shall show both $a_1c_2 + a_2c_1$ and $b_1c_2 + b_2c_1$ are $0$, which gives the contradiction $\pm 8=0$. Let us just focus on $a_1c_2 + a_2c_1$, the other case being essentially the same. If $c_1=c_2$, then \eqref{EQ:R_p_pont2} simplifies to $a_1 = -a_2$, and thus $a_1c_2 + a_2c_1 = a_1c_2 + a_2c_2 = c_2(a_1 + a_2) = 0$. If otherwise $c_1=-c_2$, then \eqref{EQ:R_p_pont2} simplifies to $a_1 = a_2$, thus $a_1 c_2 + a_2 c_1 = a_1 c_2 - a_2 c_2 = c_2(a_1 - a_2) = 0$. This completes the proof.


\item[$(1,1,1)$:] We have that $c_i\in \{\pm 1\}$, so that $\frac{1}{c_i} = c_i$ and $c_i^2 = 1$. Then the equations \eqref{EQ:R_p_pont2} and \eqref{EQ:R_p_pont3} give 
$$a_3 = -c_3(c_1 a_1 + c_2 a_2) \text{ and } b_3 = -c_3(c_1 b_1 + c_2 b_2).$$
Set $S = a_1^2 + a_2^2 + b_1^2 + b_2^2$. Substituting this in \eqref{EQ:R_p_pont1} we get 
\begin{align*} 6 + 8q^2 &= S + a_3^2 + b_3^2 + 6q^2\\ &= S + a_1^2 + 2c_1c_2a_1 a_2 + a_2^2 + b_1^2 + 2c_1c_2 b_1b_2 + b_2^2+ 6q^2 \\  &= 2S + 2c_1c_2(a_1a_2 + b_1b_2)+ 6q^2.
\end{align*}
It follows that $c_1c_2 (3+q^2 - S) = a_1 a_2+b_1b_2$. Together with the fact that $(c_1 a_1 + c_2 a_2)(c_2 a_1 + c_1a_2) = c_1c_2(a_1^2+ a_2^2) + 2a_1 a_2$, equation \eqref{EQ:R_p_euler} can be written as:
\begin{align*} \pm 8 &= c_3(a_1 a_2 + b_1b_2)  -c_3(c_1a_1+ c_2 a_2)(c_2 a_1 c_1 a_2) -c_3(c_1 b_1 + c_2b_2)(c_2 b_1 + c_1 b_2) + 2q^2 c_1 c_2 c_3\\ &= c_3(a_1 a_2 + b_1 b_2)  -c_3 c_1 c_2 S - 2c_3(a_1 a_2 + b_1 b_2) + 2q^2 c_1 c_2 c_3 \\ &= -c_1(a_1 a_2 + b_1 b_2) -c_1c_2c_3(S-2q^2)\\ &= -c_3c_1c_2(3+q^2 - S) - c_1c_2c_3(S-2q^2)\\ &= -c_1c_2c_3(3+q^2 - S + S-2q^2)\\ &= -c_1 c_2 c_3(3-q^2). 
\end{align*}
Thus, $q^2 = 3\pm 8$, giving an obvious contradiction.

\item[$(1,0,0)$:]  Assume without loss of generality that $c_1=\pm 1$ and $c_2=c_3=0$. Then from the equations \eqref{EQ:R_p_pont2} and \eqref{EQ:R_p_pont3} we get that $a_1=b_1=0$, and the equations  \eqref{EQ:R_p_euler} and \eqref{EQ:R_p_pont1} become:
$$
8= |a_2a_3 + b_2b_3|\qquad 6(q^2+1)=a_2^2+a_3^2+b_2^2+b_3^2.
$$
Combining both equations we get that $(a_2+a_3)^2 +(b_2 + b_3)^2 = 6q^2 + 6\pm 16$.  We find a contradiction by reducing mod $16$. To that end, note that a square mod $16$ equals $0,1,4,9$. Thus 
$$6q^2 + 6 \pm 16 \in \{ 6, 12, 30, 60\} = \{6,12,14\} \mod 16$$
On the other hand, by inspection, the sum $(a_2+a_3)^2 +(b_2 + b_3)^2$ of two squares mod $16$ cannot be any of $6,12$ or $14$.

\end{itemize}








\subsection{Uniqueness of reduced presentations}

In the previous subsections we have shown that for one specific presentation of each of $\mathbb{C}P^2\# \mathbb{C}P^2$, $SU(3)\bq T^2$ and the $R(p)$ biquotients, the tangent bundle is not a biquotient vector bundle.  We conclude this section with a proof that the tangent bundle cannot be a biquotient vector bundle for any presentation.

\begin{proposition}\label{prop:no_pres}
The tangent bundle of  $\mathbb{C}P^2\# \mathbb{C}P^2$, $SU(3)\bq T^2$ or the $R(p)$ biquotients with $p$ even is not a biquotient vector bundle for any presentation as a biquotient.
\end{proposition}

\begin{proof}
By Propositions~\ref{prop:cover} and~\ref{prop:reduced} it is enough to consider reduced presentations $G\bq H$ with $G$ simply connected.

For the biquotient presentations $G\bq H$ in the previous subsections, the proof that $T(G\bq H)$ is not a biquotient vector bundle for the standard presentation was done by showing that $T(G\bq H)$ is not a sum of complex line bundles.  {\color{black} From Proposition \ref{prop:biq_bundles}, it follows that $TM\notin\Bv(G\bq H)$ for any presentation with $H$ a torus.  Thus, to prove Proposition \ref{prop:no_pres}, it is sufficient to show that for any reduced presentation $G\bq H$, the group $H$ must be a torus.} 



For $\mathbb{C}P^2\# \mathbb{C}P^2$ the work of the second author \cite[Theorem 1.1]{De14} gives that the only reduced presentation has $H = T^2$.

For $SU(3)\bq T^2$,  
  notice that this space has the rational homotopy groups of $S^2\times \mathbb{C}P^2$, but that $\pi_4(SU(3)\bq T^2)$ is trivial. As shown by the second author in \cite[Proposition 4.13]{De17}, apart from quotients of $SU(3)$, any biquotient with the rational homotopy groups of $S^2\times \mathbb{C}P^2$ is a quotient of $S^3\times S^5$ by a $T^2$-action.  In particular, such examples have $\pi_4\cong \pi_4(S^3)\cong \mathbb{Z}/2\mathbb{Z}$.  Thus, the only reduced presentations of $SU(3)\bq T^2$ as $G\bq H$ have $G = SU(3)$ and $H = T^2$. 

For the $R(p)$ biquotients, \cite[Proposition 3.14 Case 1]{De17} implies that any reduced presentation must have $G = (S^3)^3$ and $H = T^3$.

\end{proof}

\section{Low dimensions}\label{S:low_dims}

This section contains the proof of Theorem~\ref{thm:dims235}. As mentioned in the Introduction, a simply connected closed biquotient of dimension at most $5$ is diffeomorphic to one of the following: $S^2,S^3,S^4$, $\mathbb{C}P^2$, $S^2\times S^2$, $\mathbb{C}P^2\# \pm\mathbb{C}P^2$, $S^5$, $S^2\times S^3$, $SU(3)/SO(3)$ or $S^3\,\widehat{\times}\, S^2$.

The fact that the tangent bundle to $\mathbb{C}P^2\# \mathbb{C}P^2$ is not a biquotient bundle for any presentation is Proposition~\ref{prop:no_pres}. The fact that $\mathbb{C}P^2$ and $\mathbb{H}P^1 = S^4$ have reduced presentations for which the tangent bundle is not a biquotient vector bundle is Theorem~\ref{thm:pres_projective_spaces}.  Thus, we must show for the remaining examples $G\bq H$ that every reduced  presentation with $G$ simply connected, $T(G\bq H)\in \Bv(G\bq H)$.   

For $S^5$ and $SU(3)/SO(3)$, it follows from \cite{KZ04,De14} that the only presentations as reduced biquotients are homogeneous, where the result is clear.

In all remaining cases, by inspection, the only non-trivial rational homotopy groups occur in degrees $2$ and $3$. From \cite[Corollary  3.5]{De17}, the fact that $G\bq H$ is reduced implies that $H$ is a torus.  From Proposition~\ref{prop:biq_bundles}, $T(G\bq H)$ is a biquotient bundle if and only if it is a sum of complex line bundles and trivial real bundles.  For $S^2 = \mathbb{C}P^1$, $S^3$ and $S^2\times S^2$, this is obviously the case.  For $\mathbb{C}P^2\# - \mathbb{C}P^2$, recall that this space is naturally an $S^2$-bundle over $S^2$.   This induces a splitting of $T(\mathbb{C}P^2 \# - \mathbb{C}P^2)$ into (real) rank $2$ subbundles, i.e., a splitting into a sum of complex line bundles.


For $S^3\,\widehat{\times}\, S^2$ recall that it has the integral cohomology ring of $S^3\times S^2$. Denote its tangent bundle by $T$ and recall that $w_2(T)$ is the unique non-zero element in $H^2(S^3\,\widehat{\times}\, S^2;\ZZ_2)$. Take any complex line bundle $L$ with $c_1(L)$ an odd multiple of a generator of $H^2(S^3\,\widehat{\times}\, S^2)=\ZZ$, so that $w_2(rL)$ of the realification $rL$ is non-trivial. Denote by $3$ the trivial real bundle of rank $3$. It clearly follows that $w_2(rL\oplus 3)=w_2(T)\neq 0$.  As $H^4(S^3\,\widehat{\times}\, S^2) = 0$, the characteristic classes $w_4$ and $p_1$ of $rL\oplus 3$ and $T$ trivially coincide. It follows from work of Čadek and Vanžura \cite[Theorem~1(ii) and Remark p. 755]{CV93} that $rL\oplus 3$ is isomorphic to $T$. By Proposition~\ref{prop:biq_bundles}, $rL\oplus 3$ is a biquotient bundle and hence so is $T$.



\begin{thebibliography}{10}

\bibitem[AGZ22]{AGZ} Amann, Manuel; González-Álvaro, David; Zibrowius, Marcus. Vector bundles of non-negative curvature over cohomogeneity one manifolds. Adv. Math. 405 (2022), Paper No. 108477, 51 pp.

\bibitem[BtD85]{BtD85} Bröcker, Theodor; tom Dieck, Tammo. Representations of compact Lie groups. Graduate Texts in Mathematics, 98. Springer-Verlag, New York, 1985.

\bibitem[CV93]{CV93} Čadek, Martin; Vanžura, Jiří. On the classification of oriented vector bundles over $5$-complexes. Czechoslovak Math. J. 43(118) (1993), no. 4, 753--764.

\bibitem[CG72]{CG72} Cheeger, Jeff; Gromoll, Detlef. On the structure of complete manifolds of nonnegative curvature. Ann. of Math. (2) 96 (1972), 413--443.

\bibitem[De14]{De14} DeVito, Jason. The classification of compact simply connected biquotients in dimensions 4 and 5. Differential Geom. Appl. 34 (2014), 128--138.

\bibitem[De17]{De17} DeVito, Jason.  The classification of compact simply connected biqoutients in dimension 6 and 7.  Math. Ann. 368 (2017), no. 304, 1493--1541.

\bibitem[DG21]{DG21} DeVito, Jason; Gonz\'alez-\'Alvaro, David. Biquotient vector bundles with no inverse. Math. Z. 302 (2022), no. 2, 1267--1278.

\bibitem[Es84]{Es84} Eschenburg, Jost-Heinrich. Freie isometrische aktionen auf kompakten Lie-gruppen mit positiv gekrümmten orbiträumen. Schriften der Math. Universität Münster, 32, 1984.

\bibitem[Es92]{Es1} Eschenburg, Jost-Heinrich.  Cohomology of biquotients.  Manuscripta Math.  75 (1992) 151--166.

\bibitem[EZ14]{EZ14} Escher, Christine; Ziller, Wolfgang. Topology of non-negatively curved manifolds. Ann. Global Anal. Geom. 46 (2014), no. 1, 23--55.

\bibitem[Ga97]{Gau97} Gauduchon, Paul. Hermitian connections and Dirac operators. Boll. Un. Mat. Ital. B (7) 11 (1997), no. 2, suppl., 257--288.

\bibitem[GZ11]{GroveZiller} Grove, Karsten; Ziller, Wolfgang.  Lifting group actions and nonnegative curvature.  Trans. Amer. Math. Soc. 363 (2011), no. 6, 2865--2890.

\bibitem[GHS82]{GHS} Glover, Henry H.; Homer, William D.; Stong, Robert E.  Splitting the Tangent Bundle of Projective Space.  Indiana Univ. Math. J.  81 (1982), no. 2, 161--166.

\bibitem[Gr67]{Gr67} Gray, Alfred. Pseudo-Riemannian almost product manifolds and submersions. J. Math. Mech. 16 (1967), 715--737. 

\bibitem[He05]{He05} Hepworth, Richard. Generalized Kreck-Stolz invariants and the topology of certain $3$-Sasakian $7$-manifolds. Ph.D. thesis at the University of Edinburgh (2005), available at \url{https://era.ed.ac.uk/handle/1842/15011}.

\bibitem[Hi66]{Hir} Hirzebruch, F., Borel, A. and Schwarzenberger, R.L.E. Topological methods in algebraic geometry (Vol. 175). Springer, Berlin, 1966.

\bibitem[Hi87]{HirCW} Hirzebruch, F. Gesammelte Abhandlungen, Band I 1951--1962 (Collected Papers, Volume 1 - 1951--1962). Springer-Verlag, 1987.

\bibitem[Hu06]{Huy} Huybrechts, D. Complex geometry: an introduction. Springer Science \& Business Media, 2006.

\bibitem[KZ04]{KZ04} Kapovitch, Vitali; Ziller, Wolfgang. Biquotients with singly generated rational cohomology. Geom. Dedicata 104 (2004), 149--160. 

\bibitem[KZ06]{KZ06} Kerin, Martin; Ziller, Wolfgang. On Eschenburg’s Habilitation on biquotient. Lecture notes, 2006, available at \url{https://www2.math.upenn.edu/~wziller/papers/SummaryEschenburg.pdf}.

\bibitem[Ke11]{Ke11} Kerin, Martin. Some new examples with almost positive curvature. Geom. Topol., (2011), 15:217--260.

\bibitem[LM89]{LM} Lawson, H.B. and Michelsohn, M.L.. Spin geometry (pms-38) (Vol. 38). Princeton university press, Princeton, NJ, 1989.

\bibitem[Ma62]{Massey} Massey, W. S. Non-existence of almost-complex structures on quaternionic projective spaces.  Pacific J. Math.  12 (1962), no. 4, 1379--1384.

\bibitem[Mi08]{Michor} Michor, P.  Topics in Differential Geometry (Graduate Studies in Mathematics), American Mathematical Society, Providence, RI, 2008.

\bibitem[Mo96]{M} Morgan, J.W. The Seiberg-Witten Equations and Applications to the Topology of Smooth Four-Manifolds (Vol. 44). Princeton University Press, 1996.

\bibitem[On66]{On1} B. O’Neill. The fundamental equations of a submersion. Michigan Math.
J., (1966) 13:459--469.

\bibitem[Pa04]{Pa04} Pavlov, A. V.  Five-dimensional biquotients.  Siberian Mathematical Journal.  45 (2004),  no. 6, 1080--1083.

\bibitem[Si93]{Si93} Singhof, W. On the topology of double coset manifolds. Math. Ann. (1993), 
297:133--146.


\bibitem[To02]{To02} Totaro, Burt. Cheeger manifolds and the classification of biquotients. J. Differential Geom. 61 (2002), no. 3, 397--451.

\end{thebibliography}
\end{document}